\documentclass[a4paper,12pt,oneside,headsepline]{scrartcl}

\usepackage{ifdraft}

\usepackage[utf8]{inputenc}
\usepackage[T1]{fontenc}

\usepackage{lmodern}
\usepackage{fourier}
\usepackage{amssymb, amsfonts}

\usepackage{mathrsfs} %
\usepackage{dsfont}
\usepackage[usenames,dvipsnames]{xcolor}%
\usepackage{lettrine}
\usepackage{url}

\usepackage{stmaryrd}

\usepackage[english]{babel}

\usepackage[draft=false]{scrlayer-scrpage}
\usepackage{wrapfig}
\usepackage{footmisc}
\usepackage{needspace}

\usepackage{amsmath}
\usepackage[thmmarks, amsmath, amsthm]{ntheorem}

\usepackage[style=numeric,giveninits,url=false,sortcites=true,maxbibnames=99]{biblatex}
\bibliography{bibliography.bib}

\ifdraft{{}}
  {\usepackage[pdftex,
               pdfborder={0 0 0}, colorlinks=true,
               linkcolor=BrickRed, citecolor=ForestGreen, urlcolor=RoyalBlue]{hyperref}}

\ifdraft{\usepackage[colorinlistoftodos]{todonotes}}{}

\usepackage{booktabs}
\usepackage[inline]{enumitem}
\usepackage{float}
\usepackage{graphicx}
\usepackage{multicol}
\usepackage{tikz}
\usetikzlibrary{matrix,arrows}

\ifdraft{\date{\today}}{\date{}}

\KOMAoptions{DIV=13}

\clearscrheadings
\automark[section]{subsection}

\renewcommand{\subsectionmark}[1]{}

\lohead{{\small \headertitle}}
\rohead{{\small \headerauthors}}

\newenvironment{plainfootnotes}{
  \deffootnote[0em]{0em}{0em}{}
}{
  \deffootnote[1em]{1.5em}{1em}{\textsuperscript{\thefootnotemark}}
}

\RedeclareSectionCommand[%
  font=\Large\sffamily\bfseries,%
  beforeskip=1\baselineskip,%
  afterskip=0.5\baselineskip,%
  indent=0em%
  ]{section}

\RedeclareSectionCommands[%
  font=\normalfont\bfseries,%
  afterskip=-1em%
  ]{subsection,subsubsection}

\RedeclareSectionCommands[%
  font=\normalfont\itshape,%
  afterskip=-1em,%
  indent=0pt,
  ]{paragraph}

\AtEveryBibitem{\clearfield{isbn}}
\AtEveryBibitem{\clearlist{language}}
\AtEveryBibitem{\clearfield{pages}}

\newenvironment{enumeratearabic}{
\begin{enumerate}[label=(\arabic*), leftmargin=0pt,labelindent=2em,itemindent=!]
}{
\end{enumerate}
}

\newenvironment{enumerateroman}{
\begin{enumerate}[label=(\roman*), leftmargin=0pt,labelindent=2em,itemindent=!]
}{
\end{enumerate}
}

\newenvironment{enumeratearabic*}{
\begin{enumerate*}[label=(\arabic*)] %
}{
\end{enumerate*}
}

\newenvironment{enumerateroman*}{
\begin{enumerate*}[label=(\roman*)] %
}{
\end{enumerate*}
}

\numberwithin{equation}{section}

\newtheorem{theoremcounter}{theoremcounter}[section]

\theoremstyle{plain}

\newtheorem{corollary}[theoremcounter]{Corollary}
\newtheorem{lemma}[theoremcounter]{Lemma}
\newtheorem{proposition}[theoremcounter]{Proposition}
\newtheorem{theorem}[theoremcounter]{Theorem}

\theoremnumbering{Roman}
\newtheorem{maintheoremcounter}{maintheoremcounter}

\newtheorem{maintheorem}[maintheoremcounter]{Theorem}

\theoremstyle{definition}

\newtheorem{definition}[theoremcounter]{Definition}

\theoremstyle{remark}

\newtheorem{example}[theoremcounter]{Example}

\newtheorem{remark}[theoremcounter]{Remark}

\newtheorem*{mainremark}{Remark}

\newtheorem*{remarkcomputation}{Computation}

\let\cal\undefined

\ifdraft{
 \newcommand{\texpdf}[2]{#1}
}{
 \newcommand{\texpdf}[2]{\texorpdfstring{#1}{#2}}
}

\newcommand{\tx}{\ensuremath{\text}}

\newcommand{\tbf}{\bfseries}

\newcommand{\thdash}{\nbd th}

\newcommand{\nbd}{\nobreakdash-\hspace{0pt}}

\newcommand{\bboard}{\ensuremath{\mathbb}}
\newcommand{\cal}{\ensuremath{\mathcal}}
\renewcommand{\frak}{\ensuremath{\mathfrak}}

\newcommand{\bbH}{\ensuremath{\bboard H}}

\newcommand{\bbM}{\ensuremath{\bboard M}}

\newcommand{\cB}{\ensuremath{\cal{B}}}
\newcommand{\cC}{\ensuremath{\cal{C}}}

\newcommand{\cH}{\ensuremath{\cal{H}}}

\newcommand{\cO}{\ensuremath{\cal{O}}}

\newcommand{\cV}{\ensuremath{\cal{V}}}

\newcommand{\frake}{\ensuremath{\frak{e}}}

\newcommand{\rmd}{\ensuremath{\mathrm{d}}}

\newcommand{\rmt}{\ensuremath{\mathrm{t}}}

\newcommand{\rmC}{\ensuremath{\mathrm{C}}}
\newcommand{\rmD}{\ensuremath{\mathrm{D}}}

\newcommand{\rmH}{\ensuremath{\mathrm{H}}}
\newcommand{\rmI}{\ensuremath{\mathrm{I}}}

\newcommand{\rmL}{\ensuremath{\mathrm{L}}}
\newcommand{\rmM}{\ensuremath{\mathrm{M}}}

\newcommand{\rmQ}{\ensuremath{\mathrm{Q}}}
\newcommand{\rmR}{\ensuremath{\mathrm{R}}}
\newcommand{\rmS}{\ensuremath{\mathrm{S}}}

\newcommand{\rmW}{\ensuremath{\mathrm{W}}}

\newcommand{\td}{\tilde}
\newcommand{\wtd}{\widetilde}
\newcommand{\ov}{\overline}
\newcommand{\ul}{\underline}
\newcommand{\wht}{\widehat}

\newcommand*{\longhookrightarrow}{\ensuremath{\lhook\joinrel\relbar\joinrel\rightarrow}}
\newcommand*{\longtwoheadrightarrow}{\ensuremath{\relbar\joinrel\twoheadrightarrow}}

\newcommand{\ra}{\ensuremath{\rightarrow}}
\newcommand{\hra}{\ensuremath{\hookrightarrow}}
\newcommand{\thra}{\ensuremath{\twoheadrightarrow}}

\newcommand{\lra}{\ensuremath{\longrightarrow}}
\newcommand{\lhra}{\ensuremath{\longhookrightarrow}}
\newcommand{\lthra}{\ensuremath{\longtwoheadrightarrow}}

\newcommand{\mto}{\ensuremath{\mapsto}}
\newcommand{\lmto}{\ensuremath{\longmapsto}}

\newcommand{\ZZ}{\ensuremath{\mathbb{Z}}}
\newcommand{\QQ}{\ensuremath{\mathbb{Q}}}
\newcommand{\RR}{\ensuremath{\mathbb{R}}}
\newcommand{\CC}{\ensuremath{\mathbb{C}}}

\renewcommand{\Im}{\ensuremath{\mathrm{Im}}}

\renewcommand{\pmod}[1]{\ensuremath{\;(\mathrm{mod}\, #1)}}

\newcommand{\Hom}{\ensuremath{\mathop{\mathrm{Hom}}}}

\newenvironment{psmatrix}{\left(\begin{smallmatrix}}{\end{smallmatrix}\right)}

\newcommand{\GL}[1]{\ensuremath{\mathrm{GL}_{#1}}}

\newcommand{\SL}[1]{\ensuremath{\mathrm{SL}_{#1}}}
\newcommand{\Mp}[1]{\ensuremath{\mathrm{Mp}_{#1}}}

\newcommand{\U}[1]{\ensuremath{\mathrm{U}_{#1}}}

\newcommand{\T}{\ensuremath{\rmt}}
\newcommand{\rT}{\ensuremath{\,{}^\T\!}}

\newcommand{\diag}{\ensuremath{\mathrm{diag}}}

\newcommand{\std}{\ensuremath{\mathrm{std}}}

\newcommand{\sym}{\ensuremath{\mathrm{sym}}}
\newcommand{\bbone}{\ensuremath{\mathds{1}}}

\newcommand{\lspan}{\ensuremath{\mathop{\mathrm{span}}}}
\newcommand{\HS}{\mathbb{H}}

\newcommand{\llangle}{\langle\!\!\langle}
\newcommand{\rrangle}{\rangle\!\!\rangle}
\newcommand{\llparen}{(\!\!(}
\newcommand{\rrparen}{)\!\!)}
\renewcommand{\tbinom}[2]{\genfrac{(}{)}{0pt}{1}{#1}{#2}}

\newcommand{\soc}{\mathrm{soc}}
\newcommand{\Poly}{\mathrm{Poly}}

\newcommand{\para}{\mathrm{pb}}
\newcommand{\Ext}{\mathrm{Ext}}
\newcommand{\Hpara}{\rmH^1_{\para}}
\newcommand{\Extpara}{\Ext^1_\para}
\newcommand{\cusp}{\mathrm{cusp}}
\newcommand{\md}{\mathrm{md}}
\newcommand{\ex}{\mathrm{ex}}
\newcommand{\Ind}{\mathrm{Ind}}
\newcommand{\Res}{\mathrm{Res}}
\newcommand{\semisimple}{\mathrm{ss}}
\newcommand{\reg}{\mathrm{reg}}

\newcommand{\Brown}{\mathrm{Brown}}
\newcommand{\shift}{\mathrm{shift}}
\newcommand{\pxs}{\mathrm{pxs}}

\newcommand{\extb}{\mathop{\overline{\boxplus}}}
\newcommand{\extbd}{\mathop{\overline{\boxplus}^\vee}}
\newcommand{\extbpara}{\mathop{\overline{\boxplus}_\para}}
\newcommand{\extbdpara}{\mathop{\overline{\boxplus}^\vee_\para}}

\newcommand{\sfd}{\mathsf{d}} %

\newcommand{\vvR}{\widetilde{\rmR}} %

\newcommand{\headertitle}{{\normalfont%
  Modular forms of virtually real-arithmetic type I
}}
\newcommand{\headerauthors}{
  M.~H.~Mertens, M.~Raum
}

\begin{document}

\begin{plainfootnotes}
\begin{flushleft}
{\fontfamily{lms}\sffamily
  \hspace{20pt}{\huge%
  Modular forms
  }\\\hspace{20pt}{\huge%
  of virtually real-arithmetic type I
  }
}
\\{\fontfamily{lms}\sffamily
  \hspace{20pt}%
  Mixed mock modular forms yield
  vector-valued modular forms
}
\\[.6em]\hspace{20pt}{\large%
  Michael H.\ Mertens
  \footnote{The research of the first named author is supported by the European Research
Council under the European Union's Seventh Framework Programme
(FP/2007-2013) / ERC Grant agreement n. 335220 - AQSER.}
  and
  Martin Raum%
  \footnote{The second named author was partially supported by Vetenskapsr\aa det Grant~2015-04139.}
}
\\[1.2em]
\end{flushleft}
\end{plainfootnotes}

\thispagestyle{scrplain}

{\small
\noindent
{\tbf Abstract:}
The theory of elliptic modular forms has gained significant momentum from the discovery of relaxed yet well-behaved notions of modularity, such as mock modular forms, higher order modular forms, and iterated Eichler-Shimura integrals. Applications beyond number theory range from combinatorics, geometry, and representation theory to string theory and conformal field theory. We unify these relaxed notions in the framework of vector-valued modular forms by introducing a new class of $\SL{2}(\ZZ)$-representations: virtually real-arithmetic types. The key point of the paper is that virtually real-arith\-metic types are in general not completely reducible. We obtain a rationality result for Fourier and Taylor coefficients of associated modular forms.
\\[.35em]
\textsf{\textbf{%
  vector-valued modular forms%
}}%
\hspace{0.3em}{\tiny$\blacksquare$}\hspace{0.3em}%
\textsf{\textbf{%
  mixed mock modular forms%
}}%
\hspace{0.3em}{\tiny$\blacksquare$}\hspace{0.3em}%
\textsf{\textbf{%
  higher order modular forms%
}}%
\hspace{0.3em}{\tiny$\blacksquare$}\hspace{0.3em}%
\textsf{\textbf{%
  iterated Eichler-Shimura integrals%
}}%
\hspace{0.3em}{\tiny$\blacksquare$}\hspace{0.3em}%
\textsf{\textbf{%
  Hecke operators%
}}%
\hspace{0.3em}{\tiny$\blacksquare$}\hspace{0.3em}%
\textsf{\textbf{%
  Eisenstein series%
}}%
\hspace{0.3em}{\tiny$\blacksquare$}\hspace{0.3em}%
\textsf{\textbf{%
  Poincar\'e series%
}}%
\\[0.15em]
\noindent
\textsf{\textbf{%
  MSC Primary:
  11F11%
}}%
\hspace{0.3em}{\tiny$\blacksquare$}\hspace{0.3em}%
\textsf{\textbf{%
  MSC Secondary:
  11F25, 11F37, 11M32
}}
}

\vspace{-1.5em}
\renewcommand{\contentsname}{}
\setcounter{tocdepth}{2}
\tableofcontents
\vspace{1.5em}

\Needspace*{4em}
\addcontentsline{toc}{section}{Introduction}
\markright{Introduction}
\lettrine[lines=2,nindent=.2em]{\tbf R}{elaxed} notions of elliptic modular forms have developed into at least three distinct research branches: Mixed mock modular forms, higher order modular forms, and iterated Eichler-Shimura integrals. Till now, they have been flourishing largely independently from one another, with exception of some connections among them summarized in the remarks following Theorem~\ref{thm:main:injection-of-modular-forms}. The raison d'\^etre of the present paper is to bring them together under the umbrella of vector-valued modular forms. This, in particular, allows us to transfer insight from one of them to the others.

There are dozens of papers on mock modular forms and mixed mock modular forms, whose modern theory has been developed since~2002~\cite{bringmann-ono-2006, bringmann-ono-2010, zwegers-2002, bruinier-funke-2004, zagier-2009, dabholkar-murthy-zagier-2018}. The range of covered topics includes combinatorics, geometry, representation theory, and string theory~\cite{ono-2009,bringmann-folsom-ono-rolen-2018}. Higher order modular forms, first introduced with this name by Chinta, Diamantis, and O'Sullivan \cite{chinta-diamantis-osullivan-2002} and Kleban and Zagier \cite{kleban-zagier-2003}, have served as a handle on the distribution of modular symbols~\cite{petridis-risager-2004,petridis-risager-2017-preprint} and their connection to the ABC-conjecture~\cite{goldfeld-2002}, and made appearance in conformal field theory~\cite{kleban-zagier-2003,diamantis-kleban-2009}. Recently, iterated Eichler-Shimura integrals have received a modular interpretation analogous to the one of mock modular forms~\cite{brown-2017b-preprint}, and at the same time were equipped with a motivic-geometric interpretation~\cite{brown-2017-preprint,baumard-schneps-2017}.

Elliptic modular form are associated to a discrete subgroup~$\Gamma \subseteq \SL{2}(\RR)$ with finite co-volume. For simplicity, we focus on the case~$\Gamma \subseteq \SL{2}(\ZZ)$. Vector-valued elliptic modular forms are functions on the Poincar\'e upper half plane $\HS$ that transform according to a given finite-dimensional, complex representation of~$\Gamma$. Such a representation~$\rho$ is called an arithmetic type. A modular form~$f$ of weight~$k$ and type~$\rho$ for~$\Gamma \subseteq \SL{2}(\ZZ)$ satisfies

\begin{gather*}
  f\big(\frac{a \tau + b}{c \tau + d}\big)
=
  (c \tau + d)^k \rho(\gamma) f(\tau)
\quad
  \tx{for all }
  \gamma
=
  \begin{psmatrix}
  a & b \\ c & d
  \end{psmatrix}
\in
  \Gamma
  \tx{ and }
  \tau
\in
  \HS
\tx{.}
\end{gather*}
We define a specific class of arithmetic types, called ``virtually real-arithmetic'' types---vra~types for short. Modular forms of vra~type subsume the relaxed notions of modular forms that we listed above in a sense that we make precise in Theorem~\ref{thm:main:injection-of-modular-forms}. Before explaining this connection to vector-valued modular forms in more detail and justify the name of vra~types, we present Theorem~\ref{thm:main:recursion-rationality} to illustrate its bifurcations. To avoid confusion at this point, let us remark that despite the similarity in name virtually real-arithmetic types are actual representations as opposed to virtual representations.

The theory of vector-valued modular forms supplies several new tools to examine mixed mock modular forms, higher order modular forms, and iterated Eichler-Shimura integrals. Two of the more recent papers in that direction, which are a good starting point, are~\cite{candelori-franc-2017,gannon-2014}. We can, for example, describe vector-valued modular forms as sections of suitable holomorphic vector bundles that carry a connection. This allows us to deduce the next rationality and recursion result. We need the following notation:
\begin{enumeratearabic}
\item
Given $k \in \ZZ$, we write $\rmM_k(\Gamma)$, $\rmM_k^!(\Gamma)$, and $\bbM_k(\Gamma)$ for the space of weight~$k$ modular forms, weakly holomorphic modular forms, and mock modular forms with shadow of moderate growth for the group~$\Gamma \subseteq \SL{2}(\ZZ)$. Then $\rmM_l(\Gamma) \otimes \bbM_k(\Gamma)$ is a space of mixed mock modular forms for every $l \in \ZZ$. Section~\ref{ssec:mixed-mock-modular-forms} contains the definitions.
\item 
Given $k \in \ZZ$ and $\sfd \in \ZZ_{\ge 0}$, we write $\rmM^{[\sfd]}_k(\Gamma)$ for the space of order~$\sfd$ modular forms of weight~$k$ for~$\Gamma$. Section~\ref{ssec:higher-order-modular-forms} contains the definition.

\item
Given $\sfd \in \ZZ_{\ge 0}$ and $\ul{k} = (k_1, \ldots, k_\sfd) \in (2 \ZZ)^\sfd$, we write $\rmI\rmM_{\ul{k}}$ for the space of iterated Eichler-Shimura integrals of level~$1$ modular forms of weights~$k_1, \ldots k_\sfd$. Section~\ref{ssec:iterated-integrals} contains the definition.
\end{enumeratearabic}
\begin{maintheorem}
\label{thm:main:recursion-rationality}
Let $f$ be
\begin{enumeratearabic}
\item
a mixed mock modular form $f \in \rmM_l(\Gamma) \otimes \bbM_k(\Gamma)$ with $l \in \ZZ$ and $k \in \ZZ_{\le 0}$,

\item
\label{it:thm:main:recursion-rationality:higher-order}
a higher order modular form $f \in \rmM^{[\sfd]}_k(\Gamma)$ with $k \in \ZZ$ and $\sfd \in \ZZ_{\ge 0}$, or

\item
an iterated Eichler-Shimura integral $f \in \rmI\rmM_{\ul{k}}$ with $\ul{k} \in (2 \ZZ)^\sfd$ and $\sfd \in \ZZ_{\ge 0}$.
\end{enumeratearabic}
Write $c(f;\,n)$, $n \in \QQ_{\ge 0}$ for its Fourier coefficients at some cusp or its Taylor coefficients at some point in the Poincar\'e upper half plane. Then the $c(f;\,n)$ satisfy a family of recursions with coefficients that span a finite-dimensional $\QQ$~vector space. In particular, we have
\begin{gather*}
  \dim_\QQ\,
  \lspan \QQ \big\{ c(f;\, n) \,:\, n \in \QQ_{\ge 0} \big\}
<
  \infty
\tx{.}
\end{gather*}
\end{maintheorem}
\begin{mainremark}
\begin{enumeratearabic}
\item
A more detailed statement that applies to all modular forms of vra~type is given in Theorem~\ref{thm:recursion-rationality}. In conjunction with Theorem~\ref{thm:main:injection-of-modular-forms}, it implies Theorem~\ref{thm:main:recursion-rationality}.

\item
The statement of Theorem~\ref{thm:main:recursion-rationality} for Fourier coefficients of mock modular forms of non-positive weight can also be inferred from Theorem~1.1 of \cite{guerzhoy-kent-ono-2010}. 

\item
As for Fourier coefficients, Theorem~\ref{thm:main:recursion-rationality} stands in stark contrast to the case of positive weight mock modular forms~\cite{duke-li-2015,ehlen-2017} and half-integral weight mock modular forms that are not mock theta functions~\cite{bruinier-ono-2010,bruinier-stromberg-2012}. For mock theta functions, rationality of Fourier coefficients follows, for example, from their description in terms of indefinite theta series~\cite{zwegers-2002}.

\item
Values at CM-points of mock modular forms in some cases not covered by Theorem~\ref{thm:main:recursion-rationality} were investigated in~\cite{duke-jenkins-2008,alfes-ehlen-2013}.

\item
The recursions for Fourier coefficients of mock theta series in~\cite{imamoglu-raum-richter-2014} were also obtained by employing vector-valued modular forms. The methods used to establish them are of spectral as opposed to geometric nature.
\end{enumeratearabic}
\end{mainremark}

Before proceeding to our discussion of vra~types, we mention one more aspect that we otherwise skip over in this paper. Namely, the theory of vector-valued modular forms has recently received attention from harmonic analysis. M\"uller established a trace formula for non-unitary twists~\cite{muller-2011}, which includes modular forms for our notion of vra~types except that M\"uller restricts to co-compact subgroups of~$\SL{2}(\RR)$. Extensions of this trace formula and the associated continuous spectrum have since then been investigated in several papers~\cite{deitmar-monheim-2016,deitmar-monheim-2016-preprint,fedosova-pohl-2017-preprint}. Theorem~\ref{thm:main:injection-of-modular-forms} can be viewed as a partial description of the holomorphic discrete spectrum for vra~types. Holomorphic modular forms have profited from a combination of geometric and spectral perspectives on them. It is desirable to employ M\"uller's trace formula in order to achieve the same in the case of mixed mock modular forms and iterated Eichler-Shimura integrals. Substantiating such hopes will require a serious effort, which we do not pursue in the present paper.

Consider the connection between classical modular forms for $\Gamma \subseteq \SL{2}(\ZZ)$ and vector-valued modular forms for~$\SL{2}(\ZZ)$. Starting with a scalar-valued modular form~$f \in \rmM_k(\Gamma)$, one constructs a function whose components are $f |_k\, \gamma$ with $\gamma$ running through a system of representatives of $\Gamma \backslash \SL{2}(\ZZ)$. The resulting modular form~$\Ind\,f$ transforms as the induced representation~$\Ind_\Gamma^{\SL{2}(\ZZ)}\,\bbone$, where $\bbone$ denotes the trivial arithmetic type for~$\Gamma$. This concept was explained in detail in~\cite{raum-2017}. It allows us to establish a bijection
\begin{gather}
\label{eq:intro:induction-map}
  \Ind :\,
  \rmM_k(\Gamma)
\lra
  \rmM_k\big( \Ind_\Gamma^{\SL{2}(\ZZ)}\,\bbone \big)
\tx{,}\quad
  f
\lmto
  \Ind\,f
\tx{.}
\end{gather}
The fact that the inverse map from $\rmM_k(\Ind_\Gamma^{\SL{2}(\ZZ)}\,\bbone)$ to~$\rmM_k(\Gamma)$ is merely the projection to one of the coordinates of~$\Ind_\Gamma^{\SL{2}(\ZZ)}\,\bbone$ means that vector-valued modular forms recover the theory of scalar-valued modular forms, but potentially expose more properties.

The bijection in~\eqref{eq:intro:induction-map} generalizes straightforwardly to the case of modular forms for Dirichlet characters~$\chi \pmod{N}$. We can view~$\chi$ as a one-dimensional representation of the congruence subgroup~$\Gamma_0(N) = \{ \begin{psmatrix} a & b \\ c & d \end{psmatrix} \in \SL{2}(\ZZ) \,:\, c \equiv 0 \pmod{N} \}$. The associated arithmetic type is~$\Ind_{\Gamma_0(N)}^{\SL{2}(\ZZ)}\,\chi$. Observe that the restrictions, if~$\Gamma$ for instance is normal,
\begin{gather*}
  \Res_\Gamma^{\SL{2}(\ZZ)}\,
  \Ind_\Gamma^{\SL{2}(\ZZ)}\,
  \bbone
\;\cong\;
  \bigoplus_{\Gamma \backslash \SL{2}(\ZZ)} \bbone
\tx{,}\quad
  \Res_{\Gamma(N)}^{\SL{2}(\ZZ)}\,
  \Ind_{\Gamma_0(N)}^{\SL{2}(\ZZ)}\,
  \chi
\;\cong\;
  \bigoplus_{\Gamma_0(N) \backslash \SL{2}(\ZZ)} \bbone
\end{gather*}
are isotrivial, that is, they are copies of the trivial representation of~$\Gamma$. Recall the terminology from, say, group theory that a property is said to hold virtually if it holds after passing to a finite index subgroup. We can therefore rephrase the bijection in~\eqref{eq:intro:induction-map}: The classical theory of modular forms is recovered by vector-valued modular forms of virtually isotrivial arithmetic type for~$\SL{2}(\ZZ)$.

To clarify the particular use of vector-valued modular forms as opposed to mere scalar valued ones, we reexamine~\eqref{eq:intro:induction-map} from the perspective of Fourier expansions. Given the Fourier expansion of a modular form on the right-hand side of~\eqref{eq:intro:induction-map}, it is easy to pass to the left-hand side by picking the component of $\Ind\,f$ that corresponds to the trivial coset in $\Gamma \backslash \SL{2}(\ZZ)$. Passage from the left to the right-hand side is intricate, if the Fourier expansion of $f$ is only given at the cusp~$\infty$, as is commonly the case. Indeed, the Fourier expansion at~$\infty$ of all $f \big|_k\, \gamma$ depends on Fourier expansions of~$f$ at all cusps. They are generally difficult to determine from the one at infinity. In~\cite{raum-2017}, one of the authors illustrated how to obtain them by employing the vector-valued point of view on modular forms. Other applications, in which the vector-valued point of view is decisive can, for instance, be found in~\cite{scheithauer-2009,carnahan-2012}.

When examining vector-valued modular forms, it is natural to consider the standard representation~$\std$ of~$\SL{2}(\ZZ)$ and its symmetric powers~$\sym^\sfd$. They are the only irreducible arithmetic types that extend to~$\SL{2}(\RR)$. Kuga and Shimura~\cite{kuga-shimura-1960} studied this case and offered a description of modular forms of type~$\sym^\sfd$. They concluded under certain assumptions on $k$ and $\sfd$ that
\begin{gather*}
  \rmM_k(\sym^\sfd)
\;\cong\;
  \bigoplus_{\substack{j = -\sfd \\ j\equiv \sfd\pmod 2}}^\sfd
  \rmM_{k+j}\big( \SL{2}(\ZZ) \big)
\tx{.}
\end{gather*}
The isomorphism is given by a linear holomorphic differential operator that is compatible with induction from subgroups. We can therefore extend our previous statement from virtually isotrivial arithmetic types to arithmetic types that virtually extend to~$\SL{2}(\RR)$. Modular forms of such type can be described in terms of classical ones.

Induction from finite index subgroups and the isomorphism due to Kuga-Shimura together exhaust the available connections between scalar- and vector-valued modular forms. Note that all arithmetic types discussed until now are completely reducible. The key point of this paper is to consider vector-valued modular forms of types that are not completely reducible.

We say that an arithmetic type is real-arithmetic if its semi-simplification extends to $\SL{2}(\RR)$. It is virtually real-arithmetic (vra) if its restriction to some finite index subgroup of~$\SL{2}(\ZZ)$ is real-arithmetic. An explicit non-trivial example of a vra~type can be found in Examples~\ref{ex:vra-types} and~\ref{ex:vra-type-delta}.

Our main theorem states that mixed mock modular forms, higher order modular forms, and iterated Eichler-Shimura integrals are components of modular forms of vra~type. We use the notation established before Theorem~\ref{thm:main:recursion-rationality}.
\begin{maintheorem}
\label{thm:main:injection-of-modular-forms}
We have the following injections into spaces of modular forms of virtually real-arithmetic type:
\begin{enumeratearabic}
\item
\label{it:thm:main:injection-of-modular-forms:mock}
For $l \in \ZZ$ and $k \in \ZZ_{\le 0}$, there is an injection
\begin{gather*}
  \rmM_l^!(\Gamma)
  \otimes
  \bbM_k(\Gamma)
\lhra
  \rmM^!_{k+l}(\rho)
\end{gather*}
for a vra~type~$\rho$ that only depends on~$k$ and~$\Gamma$. There is a retract that is a coordinate projection. Details are given in Theorem~\ref{thm:mock-modular-as-co-universal-type-modular-form} and Corollaries~\ref{cor:mock-modular-as-universal-type-modular-form} and~\ref{cor:mixed-mock-modular-as-universal-type-modular-form}.

\item
\label{it:thm:main:injection-of-modular-forms:higher-order}
Given $\sfd \in \ZZ_{\ge 0}$, there is an injection
\begin{gather*}
  \rmM_k^{[\sfd]}(\Gamma)
\lhra
  \rmM_k(\rho)
\end{gather*}
for a vra~type~$\rho$ that only depends on~$\sfd$ and~$\Gamma$. There is a retract that is a coordinate projection. Details are given in Proposition~\ref{prop:higher-order-modular-forms}.

\item
\label{it:thm:main:injection-of-modular-forms:iterated-integrals}
Given $\sfd \in \ZZ_{\ge 0}$ and $\ul{k} \in \ZZ^\sfd$ there is an injection
\begin{gather*}
  \rmI\rmM_{\ul{k}}
\lhra
  \rmM_0(\rho)
\end{gather*}
for a vra~type~$\rho$ that only depends on~$\ul{k}$. Details, including a substitute for the retracts in case~\ref{it:thm:main:injection-of-modular-forms:mock} and \ref{it:thm:main:injection-of-modular-forms:higher-order}, are given in Proposition~\ref{prop:iterated-integrals-modular}.
\end{enumeratearabic}
\end{maintheorem}
\begin{mainremark}
\begin{enumeratearabic}
\item
The representations~$\rho$ in Theorem~\ref{thm:main:injection-of-modular-forms} are not necessarily unique. For example, for mixed mock modular forms we provide a whole family of embeddings.

\item
The restriction to non-positive, integral weight in part~\ref{it:thm:main:injection-of-modular-forms:mock} of Theorem~\ref{thm:main:injection-of-modular-forms} excludes all mock theta functions, treated in~\cite{zwegers-2002}. It is possible to subsume them under a formally similar notion. This, however, would require the use of infinite-dimensional representations, namely holomorphic representations for the metaplectic double cover of~$\SL{2}(\RR)$.

\item
Case~\ref{it:thm:main:injection-of-modular-forms:mock} in conjunction with case~\ref{it:thm:main:injection-of-modular-forms:higher-order} of Theorem~\ref{thm:main:injection-of-modular-forms} suggests that there is a relation between second order modular forms and mixed mock mock modular forms whose mock-part has weight~$0$. Such a connection has been previously revealed in~\cite{bringmann-kane-2013}.

\item
Connections between mock modular forms and classical Eichler integrals are well-es\-tab\-lished~\cite{guerzhoy-2008, bringmann-guerzhoy-kent-ono-2013}. These, however, are of a somewhat different flavor than the one suggested by Cases~\ref{it:thm:main:injection-of-modular-forms:mock} and~\ref{it:thm:main:injection-of-modular-forms:iterated-integrals} of Theorem~\ref{thm:main:injection-of-modular-forms}.

\item
Case~\ref{it:thm:main:recursion-rationality:higher-order} in conjunction with case~\ref{it:thm:main:injection-of-modular-forms:iterated-integrals} of Theorem~\ref{thm:main:injection-of-modular-forms} suggests that there is a relation between higher order modular forms and iterated integrals. Such a connection has been previously revealed in~\cite{diamantis-sreekantan-2006}.

\item
Apparently unaware of either~\cite{bringmann-kane-2013} or its rephrasing in terms of vector-valued modular forms, the authors of~\cite{candelori-harland-marks-yepez-2017-preprint} reproduce~\cite{bringmann-kane-2013} in their Theorem~1.4. It is also a special case of our Proposition~\ref{prop:higher-order-modular-forms}.

\item
Indecomposability for ``analytic types'' has been previously observed, studied, and employed in~\cite{schulze-pillot-2011,bringmann-kudla-2016-preprint,raum-2018} in the context of harmonic weak Maa\ss\ forms. Since mock modular forms are tied to harmonic weak Maa\ss\ forms via a procedure called modular completion, Theorem~\ref{thm:main:injection-of-modular-forms} may be interpreted as evidence that analytic and arithmetic indecomposability are connected to each other.
\end{enumeratearabic}
\end{mainremark}

The proof of Theorem~\ref{thm:main:injection-of-modular-forms} is contained in Section~\ref{sec:other-notions-of-modular-forms}. It is preceded by a general treatment of vra~types in Section~\ref{sec:virtually-real-arithmetic-types}. We then proceed to Section~\ref{sec:existence-of-modular-forms}, in which we establish existence of modular forms of vra~type if the weight is sufficiently large. We first give a cohomological argument in complete analogy with the one in~\cite{bruinier-funke-2004}, which concerned harmonic weak Maa\ss\ forms. Second, in Section~\ref{ssec:eisenstein-poincare-series}, we give an argument via Poincar\'e series and provide an explicit lower bound on the weight required for their convergence.

In the forthcoming part~II of this work, we plan to focus on computational applications of our results, in particular exploiting the module structure of modular forms of vra~type.

\paragraph{Acknowledgment}

The second author is grateful to Steve Kudla for mentioning to him the Kuga-Shimura paper at the 2016 Oberwolfach conference on modular forms, which later inspired the notion of vra~types. Both authors are grateful to Morten Risager for his comments on an earlier version of the manuscript. They also wish to express their gratitude to the Max-Planck-Institute for Mathematics in Bonn. The authors thank the referee for a thorough reading which uncovered some glitches in an earlier version of the manuscript.

\section{Preliminaries}

We summarize notation and preliminaries on the modular group, arithmetic types, and modular transformations. We also revisit required terminology from representation theory and the theory of harmonic weak Maa\ss\ forms.

\subsection{Upper half space and modular group}

The set
\begin{gather*}
  \HS
:=
  \big\{
  \tau \in \CC \,:\,
  y = \Im\tau > 0
  \big\}
\end{gather*}
is called the Poincar\'e upper half space. It carries an action of $\SL{2}(\RR)$ via M\"obius transformations
\begin{gather*}
  g \tau
=
  \frac{a \tau + b}{c \tau + d}
\tx{,}\quad
  g
=
  \begin{psmatrix} a & b \\ c & d \end{psmatrix}
\tx{.}
\end{gather*}
We let $S = \begin{psmatrix} 0 & -1 \\ 1 & 0 \end{psmatrix}$ and $T = \begin{psmatrix} 1 & 1 \\ 0 & 1 \end{psmatrix}$. Throughout, we use the notation
\begin{gather*}
  \Gamma_\infty
:=
  \big\{
  \begin{psmatrix} a & b \\ 0 & d \end{psmatrix} \in \SL{2}(\ZZ)
  \big\}
\tx{,}
\qquad
  \Gamma(N)
:=
  \big\{
  \begin{psmatrix} a & b \\ c & d \end{psmatrix} \in \SL{2}(\ZZ)
  \,:\,
  a,d \equiv 1 \pmod{N},\, b,c \equiv 0 \pmod{N}
  \big\}
\tx{.}
\end{gather*}

\subsection{Weights and slash actions}

For the purpose of this paper, we identify weights~$k \in \ZZ$ with one-dimensional representations 
\begin{gather*}
  \sigma_k(g)
:=
  g^k
\tx{,}\quad
  g \in \GL{1}(\CC)
\tx{.}
\end{gather*}
For any weight~$k$, we obtain a slash action
\begin{gather}
\label{eq:def:slash-action}
  \big(f \big|_k g\big)(\tau)
:=
  \sigma_k\big((c \tau + d)^{-1}\big)\, f(g \tau)
=
  (c \tau + d)^{-k}\, f(g\tau)
\tx{,}\quad
  g \in \SL{2}(\RR)
\end{gather}
on functions $f :\, \HS \ra \CC$.

\subsection{Arithmetic types}

We call a finite-dimensional, complex representation of some finite index subgroup~$\Gamma \subseteq \SL{2}(\ZZ)$ an arithmetic type. Given an arithmetic type~$\rho$, we typically denote its representation space by~$V(\rho)$. In special cases we identify a representation and its representation space.

Arithmetic types~$\rho$ in conjunction with weights~$k$ yield further slash actions on functions $f:\bbH\to V(\rho)$, 
\begin{gather}
  \big( f \big|_{k,\rho}\, \gamma \big)(\tau)
:=
  \rho(\gamma^{-1})
  \sigma_k\big((c \tau + d)^{-1}\big)\, f(\gamma \tau)
\tx{,}\quad
  \gamma \in \Gamma
\tx{.}
\end{gather}
Observe that we have to restrict to $\gamma \in \Gamma$, since $\rho(\gamma^{-1})$ is not defined for more general~$\gamma$.

We write $\bbone$ for the trivial arithmetic type. The dual of an arithmetic type~$\rho$ is denoted by~$\rho^\vee$. The evaluation $w^\vee(v)$ of $v \in V(\rho)$ and $w^\vee \in V(\rho^\vee)$ will occasionally be written as~$\langle v, w^\vee \rangle$.

\subsection{Symmetric powers}
\label{ssec:symmetric-powers}

Complex, irreducible, finite-dimensional representations of $\SL{2}(\RR)$ are exhausted by the symmetric powers~$\sym^{\sfd}$ of the standard representation~$\std$. Models for $\sym^{\sfd}$ are provided by the spaces~$\Poly(X, \sfd)$ of polynomials in~$X$ of degree at most~$\sfd$ with complex coefficients via
\begin{gather}
  \sym^\sfd(\gamma) p(X)
:=
  p(X) \big|_{-\sfd} \gamma^{-1}
=
  (-c X + a)^\sfd\,
  p\Big( \frac{d X - b}{-c X + a} \Big)
\tx{.}
\end{gather}
Observe that using this model, the slash action of type~$\sym^\sfd$ takes the simple form:
\begin{gather*}
  \big( f \big|_{k, \sym^\sfd}\, \gamma \big) (X; \tau)
=
  (c X + d)^{\sfd}
  (c \tau + d)^{-k}\,
  f\Big( \frac{a X + b}{c X + d},\, \frac{a \tau + b}{c \tau + d} \Big)
\tx{.}
\end{gather*}

We record that the trivial subrepresentation of $\Poly(X,\sfd) \otimes \Poly(Y,\sfd)$ is spanned by
\begin{gather*}
  (X - Y)^\sfd
\;=\;
  \sum_{i = 0}^\sfd
  \tbinom{\sfd}{i}  X^i (-Y)^{\sfd-i}
\tx{.}
\end{gather*}
In particular, we obtain an isomorphism
\begin{gather}
\label{eq:sym-self-duality}
  \Poly(X,\sfd)^\vee
\lra
  \Poly(Y,\sfd)
\tx{,}\quad
  X^{i\,\vee}
\lmto
  \tbinom{\sfd}{i}
  (-Y)^{\sfd-i}
\tx{.}
\end{gather}
It can be used to define the $\SL{2}(\RR)$-invariant pairing
\begin{gather}
\label{eq:def:sym-power-pairing}
  \langle\,\cdot\,,\,\cdot\,\rangle :\,
  \Poly(X, \sfd) \otimes \Poly(Y, \sfd)
\lra
  \CC
\tx{,}\quad
  \langle X^i, Y^j \rangle
\;:=\;
  \begin{cases}
  (-1)^j \big\slash \tbinom{\sfd}{i}
  \tx{,} & \tx{if $i + j = \sfd$;}
  \\
  0\tx{,} & \tx{otherwise.}
  \end{cases}
\end{gather}
We will identify~$\Poly(X,\sfd)$ and its dual via this isomorphism without further mentioning it. Note also that for $p \in \Poly(X,\sfd)$, we have 
\begin{gather}
\label{eq:sym-power-pairing-variable-substitution}
  \big\langle p(X),\, (Z - Y)^\sfd \big\rangle
=
  p(Z)
\tx{.}
\end{gather}

The Clebsch-Gordan rules assert that there is an embedding
\begin{gather}
\label{eq:def:sym-power-product-embedding}
  \Poly(Z, \sfd + \sfd')
\lhra
  \Poly(X, \sfd)
  \otimes
  \Poly(Y, \sfd')
\tx{,}\quad
  Z^n
\lmto
  \sum_{\substack{0 \le i \le \sfd \\ 0 \le j \le \sfd' \\ i+j = n}}
  \tbinom{\sfd}{i}
  \tbinom{\sfd'}{j}
  \tbinom{\sfd+\sfd'}{n}^{-1}\,
  X^i Y^j
\tx{.}
\end{gather}
The formula can be found by starting with the canonical projection from~$\Poly(X, \sfd) \otimes \Poly(Y, \sfd')$ onto~$\Poly(Z, \sfd + \sfd')$, mapping both $X$ and~$Y$ to~$Z$. Dualizing this surjection and applying the self-duality isomorphism in~\eqref{eq:sym-self-duality} yields~\eqref{eq:def:sym-power-product-embedding}.

\subsection{Restriction and induction of arithmetic types}
\label{ssec:induction-of-types}

The restriction of an arithmetic type from~$\Gamma$ to $\Gamma' \subseteq \Gamma$ is denoted by~$\Res_{\Gamma'}^\Gamma\,\rho$.

Induction of arithmetic types allows us to focus on the case $\Gamma = \SL{2}(\ZZ)$, if we wish. A short exposition is contained in Section~2 of~\cite{raum-2017}. We fix $\Gamma' \subseteq \Gamma \subseteq \SL{2}(\ZZ)$ and an arithmetic type~$\rho$ for~$\Gamma'$. Then for a fixed set~$B$ of representatives of $\Gamma' \backslash \Gamma$ we set
\begin{gather*}
  V\big( \Ind_{\Gamma'}^{\Gamma}\, \rho \big)
:=
  V(\rho) \otimes \CC[B]
\tx{,}\quad
  \Ind\, \rho(\gamma) (v \otimes \frake_\beta)
:=
  \rho\big( I_\beta^{-1}(\gamma^{-1}) \big)v
  \otimes
  \frake_{\overline{\beta \gamma^{-1}}}
\tx{,}
\end{gather*}
where $I_\beta(\gamma) \ov{\beta \gamma} = \beta \gamma$ with $I_\beta(\gamma) \in \Gamma'$ and $\ov{\beta \gamma} \in B$. The space of modular forms for $\rho$ and $\Gamma'$ is isomorphic to the space of modular forms for~$\Ind\,\rho$ and $\Gamma$. 

Recall that there is a canonical inclusion
\begin{gather}
\label{eq:inclusion-induction-restriction}
  \rho
\lhra
  \Ind_{\Gamma'}^{\Gamma}\;
  \Res_{\Gamma'}^{\Gamma}\,
  \rho
\end{gather}
for all finite index subgroups~$\Gamma' \subseteq \Gamma$ and arithmetic types~$\rho$ for~$\Gamma$.

\subsection{Socle series}
\label{ssec:socle-series}

Recall that the socle $\soc(\rho)$ of a representation~$\rho$ is the intersection of its essential submodules. The socle series of a representation is $\soc^0(\rho) \subset \cdots \subset \soc^j(\rho)$, where $j$ is called the socle length of $\rho$, and
\begin{enumerateroman}
\item $\soc^0(\rho) = \{ 0 \}$,
\item $\soc^i(\rho) \big\slash \soc^{i-1}(\rho) \;=\; \soc\big( \rho \big\slash \soc^{i-1}(\rho) \big)$ for all $1 \le i \le j$.
\end{enumerateroman}
We call
\begin{gather*}
  \soc^{i-1}(\rho)
\lhra
  \soc^i(\rho)
\lthra
  \soc^i(\rho) \big\slash \soc^{i-1}(\rho)
\end{gather*}
the $i$\thdash\ socle extension of~$\rho$. The semi-simplification of~$\rho$ is denoted by
\begin{gather*}
  \rho^\semisimple
\;:=\;
  \bigoplus_{i=1}^j \soc^i(\rho) \big\slash \soc^{i-1}(\rho)
\tx{.}
\end{gather*}

Aligning with the common notion of depth and order for mock modular forms and higher order modular forms, respectively, we say that $\rho$ has depth~$j-1$ if it has socle length~$j$.

By slight abuse of terminology, we say that the vector space homomorphism~$\rho(\gamma)$ has socle length $j$, i.e.\ depth~$j-1$, if the $\ZZ$-re\-pre\-sen\-ta\-tion $i \mto \rho(\gamma)^i$ has socle length~$j$. Similarly, we say that $\rho(\gamma)$ is unitarizable if $i \mto \rho(\gamma)^i$ is unitarizable, i.e.\ $\rho(\gamma)$ is diagonizable with eigenvalues of absolute value~$1$.

\subsection{Extensions of arithmetic types}

Given arithmetic types~$\rho$ and $\rho'$ for~$\Gamma \subseteq \SL{2}(\ZZ)$, the space of extension classes for exact sequences
\begin{gather*}
  \rho \lhra \rho'' \lthra \rho'
\end{gather*}
is denoted by $\Ext^1(\rho', \rho)$, suppressing $\Gamma$ from the notation. The set of parabolic extension classes is denoted by $\Extpara(\rho', \rho) \subseteq \Ext^1(\rho', \rho)$. Observe that we have
\begin{gather*}
  \Ext^1\big(\rho', \rho\big)
\;\cong\;
  \rmH^1\big( \Gamma,\, \rho^{\prime\,\vee} \otimes \rho \big)
\quad\tx{and}\quad
  \Extpara\big(\rho', \rho\big)
\;\cong\;
  \Hpara\big( \Gamma,\, \rho^{\prime\,\vee} \otimes \rho \big)
\tx{.}
\end{gather*}
In particular, parabolic extension classes correspond to parabolic cocycles in the first isomorphism, which can be computed using modular forms~\cite{pasol-popa-2013}.

Given $\varphi \in \Ext^1(\rho', \rho)$ , we denote by $\rho \boxplus_\varphi \rho'$ the extension that corresponds to $\varphi$, fitting into the short exact sequence
\begin{gather*}
  \rho \lhra \rho \boxplus_\varphi \rho' \lthra \rho'
\tx{.}
\end{gather*}
Specifically, we have
\begin{gather*}
  \big( \rho \boxplus_\varphi \rho' \big)(\gamma)\,
  \big(v \boxplus v'\big)
=
  \big( \rho(\gamma) v + \varphi(\gamma) v' \big)
  \,\boxplus\,
  \rho'(\gamma) v'
\tx{.}
\end{gather*}

An arithmetic type is called parabolic if all its socle extensions are parabolic. Given a cocycle $\varphi \in \rmH^1(\Gamma, \rho)$ or an extension class~$\varphi \in \Ext^1(\rho',\rho)$, we let
\begin{gather}
  \pxs(\varphi)
:=
  \begin{cases}
  0\tx{,} & \tx{if $\varphi$ is parabolic;}
  \\
  1\tx{,} & \tx{otherwise;}
  \end{cases}
\end{gather}
be its parabolic excess. I.e.\ the parabolic excess of the socle extension classes of a parabolic arithmetic type is~$0$.

\subsection{Function spaces with slash actions}

We write $\cC^\infty(\HS)$ for the space of smooth functions on $\HS$ that takes values in $\CC$. Generalizing this notation, we write $\cC^\infty(\HS \ra V)$ for the space of smooth functions on $\HS$ that take values in a complex, finite-dimensional vector space $V$. Given a weight~$k \in \ZZ$ and an arithmetic type~$\rho$, we let $\cC^\infty_k(\rho) := \cC^\infty(\HS \ra V(\rho))$ and equip it with the $\Gamma$-slash action of weight~$k$ and type~$\rho$.

We say that $f \in \cC^\infty_k(\rho)$ has moderate growth if there is some $a \in \RR$ such that for every $\gamma \in \SL{2}(\ZZ)$ we have
\begin{gather*}
  \big( f \big|_k\, \gamma \big) (\tau)
=
  \cO(y^a)
\quad
  \tx{as $y \ra \infty$.}
\end{gather*}
The space of such functions is denoted by $\cC^\infty_k(\rho)^\md$. As we restrict to finite-dimensional arithmetic types, this definition is independent of a choice of norm on $V(\rho)$. 

The spaces~$\cC^\infty_k(\rho)^\ex$ and~$\cC^\infty_k(\rho)^\cusp$ of smooth functions of at most exponential growth and smooth cuspidal functions are defined analogously by the conditions
\begin{gather*}
  \big( f \big|_k\, \gamma \big) (\tau)
=
  \cO(\exp(ay))
\quad
  \tx{as $y \ra \infty$}
\end{gather*}
and
\begin{gather*}
  \big( f \big|_k\, \gamma \big) (\tau)
\lra
  0
\quad
  \tx{as $y \ra \infty$,}
\end{gather*}
respectively.

We write $\cH_k(\rho)$, $\cH_k(\rho)^\md$, $\cH_k(\rho)^\ex$, and $\cH_k(\rho)^\cusp$ for the corresponding spaces of holomorphic functions.

\subsection{Cocycles attached to functions}
\label{ssec:cocycles-attached-to-functions}

Given~$f \in \cC^\infty_k(\rho)$ for $k \in \ZZ$ and an arithmetic type~$\rho$, the function
\begin{gather}
\label{eq:def:f-cocycle-trival-type}
  \varphi_f(\gamma)
:=
  f \big|_k\, (\gamma-1)
\tx{,}\quad
  \gamma \in \SL{2}(\ZZ)
\end{gather}
is a $1$-cocycle, which a priori takes values in $\cC^\infty_k(\rho)$.

\subsection{Modular forms}

Let $k \in \ZZ$ and fix an arithmetic type~$\rho$ for~$\Gamma \subseteq \SL{2}(\ZZ)$. A function~$f \in \cH_k(\rho)^\ex$ is called a weakly holomorphic modular form of weight~$k$ and (arithmetic) type~$\rho$ if $f \big|_{k,\rho}\, \gamma = f$ for all $\gamma \in \Gamma$. If $f \in \cH_k(\rho)^\md$ it is called a modular form. If $f \in \cH_k(\rho)^\cusp$ it is called a cusp form.

The spaces of such weakly holomorphic modular forms, modular forms, and cusp forms are denoted by $\rmM^!_k(\rho)$, $\rmM_k(\rho)$, and $\rmS_k(\rho)$, respectively. If $\rho$ is the trivial type~$\bbone$, we occasionally suppress it from the notation or replace it by~$\Gamma$.

We write
\begin{gather}
  \rmM^!_\bullet(\rho)
:=
  \bigoplus_k \rmM^!_k(\rho)
\quad\tx{and}\quad
  \rmM_\bullet(\rho)
:=
  \bigoplus_k \rmM_k(\rho)
\end{gather}
for the graded module of (weakly holomorphic) modular forms. In the case of $\rho = \bbone$ it is the graded ring of (weakly holomorphic) modular forms for~$\Gamma$.

Let $k \in \ZZ$ and fix an arithmetic type~$\rho$ for~$\Gamma \subseteq \SL{2}(\ZZ)$.
We define quasimodular forms of depth~$0$ as modular forms in the usual sense. A function~$f \in \cH_k(\rho)^\md$ is called a quasimodular form of weight~$k$, (arithmetic) type~$\rho$, and depth~$\sfd > 0$ if
\begin{gather*}
  \varphi_f\big( \begin{psmatrix} a & b \\ c & d \end{psmatrix} \big)
=
  \sum_{t = 1}^{\sfd}
  \frac{c^t f_t(\tau)}{(c\tau + d)^t}
\end{gather*}
for quasimodular forms~$f_t$ of depth~$t$. The space of such functions is denoted by $\rmQ\rmM^{[\sfd]}_k(\rho)$. Note that the definition of quasimodular forms is recursive.
\begin{remark}
We will later introduce the space $\rmM^{[\sfd]}_k(\rho)$ of modular forms of order~$\sfd$, which should not be confused with~$\rmQ\rmM^{[\sfd]}_k(\rho)$.
\end{remark}

An almost holomorphic modular form is a function $f \in \cH_k(\rho)^\md[y^{-1}]$  that transforms like a modular form. Quasimodular forms are linked to almost holomorphic modular forms via modular completions: Every quasimodular forms appears as the constant term with respect with respect to $y^{-1}$ of an almost holomorphic modular form, and vice versa the constant term with respect to~$y^{-1}$ of an almost holomorphic modular form is a quasimodular form.

\subsection{Fourier and Taylor expansions}

Given an arithmetic type~$\rho$ with unitarizable $\rho(T)$, any $f \in \rmM^!_k(\rho)$ admits a Fourier expansion of the form
\begin{gather}
  f(\tau)
=
  \sum_{n \in \QQ} c(f;\, n)\, e(n \tau)
\tx{,}\quad
  c(f;\,n) \in V(\rho)
\tx{,}
\end{gather}
where here and throughout we use the notation $e(x) = \exp(2 \pi i\, x)$. For general~$\rho$, $f \in \rmM^!_k(\rho)$ has a Fourier expansion of the form
\begin{gather}
  f(\tau)
=
  \sum_{n \in \QQ} c(f;\, n;\, \tau) e(n \tau)
\tx{,}\quad
  c(f;\, n;\, \tau) \in V(\rho) \otimes \Poly(\sfd, \tau)
\tx{,}
\end{gather}
where~$\sfd$ is the depth of~$\rho(T)$.
\begin{remark}
To see that $c(f;\, n;\, \tau)$ is indeed a polynomial, it suffices to identify the finite-dimensional, indecomposable subrepresentations~$\sigma \subset \cH \circlearrowleft T^\bullet = \{ T^n \,:\, n \in \ZZ \}$. We may assume that $c_0(\tau), \ldots, c_\sfd(\tau)$ is a basis of $V(\sigma)$ on which~$T$ acts in Jordan normal form. The characters of~$T^\bullet$ are $T \mto e(m)$ for $m \in \RR$, realized by $e(m \tau) \in \cH_k(\bbone)$. Twisting by a suitable character, we can assume that $c_0(\tau) | T = c_0(\tau)$. So $c_0(\tau)$ is periodic, and hence has a usual Fourier series expansion. Induction on $0 < i \le \sfd$ shows that $c_i(\tau + 1) = c_i(\tau) | T = c_i(\tau) + c_{i-1}(\tau)$. The solution of this difference equation is unique up to addition of a periodic function. It is a polynomial in $\tau$ of degree~$i$ whose coefficients are periodic functions. Observe that this also reveals that the $c(f;\, n;\, \tau)$ are determined by their constant coefficients with respect to~$\tau$.
\end{remark}

Given $\tau_0 \in \HS$, we write the Taylor expansion of $f \in \rmM^!_k(\rho)$ as
\begin{gather*}
  f(\tau)
=
  \sum_{n = 0}^\infty
  \frac{c_{\tau_0}(f;\, n)}{n!\, (2 \pi i)^n}\,
  \big( \tau - \tau_0 \big)^n
\tx{,}\quad
  c_{\tau_0}(f;\,n)
:=
  (2 \pi i)^n
  \frac{\rmd^n f}{\rmd\tau^n}(\tau_0)
\in
  V(\rho)
\tx{.}
\end{gather*}

\subsection{Harmonic weak Maa\ss\ forms}
\label{ssec:harmonic-weak-maass-forms}

Recall the Maa\ss\ lowering and raising operators
\begin{gather*}
  \rmL_k := \rmL
:=
  -2iy^2 \partial_{\ov\tau}
\tx{,}\quad
  \rmR_k
:=
  2i\partial_\tau +k y^{-1}
\tx{.}
\end{gather*}
Their composition $\Delta_k := -\rmR_{k-2} \rmL_k$ is the weight~$k$ hyperbolic Laplace operator.

Let $k \in \ZZ$ and fix an arithmetic type~$\rho$ for~$\Gamma \subseteq \SL{2}(\ZZ)$. A function $f \in \cC^\infty_k(\rho)^\ex$ is called a harmonic weak Maa\ss\ form of weight~$k$ and (arithmetic) type~$\rho$ if $\Delta_k\, f = 0$ and $f \big|_{k,\rho}\, \gamma = f$ for all $\gamma \in \Gamma$. In this paper, we restrict to mock modular forms that arise from harmonic weak Maa\ss\ form in the sense of~\cite{bruinier-funke-2004}, i.e.\ harmonic weak Maa\ss\ forms~$f$ for which $\rmL_k\,f \in \cC^\infty_{k-2}(\rho)^\md$ has moderate growth.

Every harmonic weak Maa\ss\ form~$f$ admits a decomposition as the sum of a holomorphic part~$f^+$ and a non-holomorphic part~$f^-$: We have~$f = f^+ + f^-$. If $k \le 0$, $\rho(T)$ is unitarizable, and $\rmL_k\, f$ has moderate growth, then
\begin{gather}
\label{eq:decomposition-into-holomorphic-part}
\begin{aligned}
  f^+(\tau)
&=
  \sum_{n \in \QQ} c^+(f;\, n)\, e(n \tau)
\qquad\tx{and}
\\
  f^-(\tau)
&=
  c^-(f;\, 0) y^{1-k}
  +
  \sum_{n \in \QQ_{< 0}} c^-(f;\, n)\, \Gamma(1-k,4 \pi |n| y) e(n \tau)
\tx{,}
\end{aligned}
\end{gather}
where $\Gamma(s, y)$ denotes the upper incomplete gamma function.

\subsection{Mixed mock modular forms}
\label{ssec:mixed-mock-modular-forms}

Let $k \in \ZZ$ and fix an arithmetic type~$\rho$ for~$\Gamma \subseteq \SL{2}(\ZZ)$. A function $f \in \cH_k(\rho)^\ex$ is called a mock modular form of weight~$k$ and (arithmetic) type~$\rho$ if there is a harmonic weak Maa\ss\ form~$\wtd{f}$ of weight~$k$ and type~$\rho$ such that $f = \wtd{f}^+$. The space of such mock modular forms is denoted by $\bbM_k(\rho)$. As in the case of modular forms, if $\rho$ is the trivial type~$\bbone$, we occasionally suppress it from the notation or replace it by~$\Gamma$.

Given another $l \in \ZZ$ and another arithmetic type~$\rho'$, we call
\begin{gather}
\label{eq:def:mixed-mock-modular-forms}
 \rmM_l^!(\rho')
 \otimes
 \bbM_k(\rho)
\end{gather}
the space of mixed mock modular forms of weight~$(k,l)$ and type $\rho' \otimes \rho$. This coincides with the notion in~\cite{dabholkar-murthy-zagier-2018}.

If $k \in 2\ZZ_{\le 0}$ and $\rho$ has finite index kernel, then by Fay~\cite{fay-1977} and Bruinier-Funke~\cite{bruinier-funke-2004}, the cocycle $\varphi_f$ attached to a mock modular form of weight~$k$ and type~$\rho$ takes values in $\Poly(|k|, \tau) \otimes V(\rho)$. Consequentially, the cocycles~$\varphi_f$ for mixed mock modular forms~$f$ in~\eqref{eq:def:mixed-mock-modular-forms} take values in $\rmM_l(\rho') \otimes  \Poly(|k|, \tau) \otimes V(\rho)$.

\section{Virtually real-arithmetic types}
\label{sec:virtually-real-arithmetic-types}

The notion of virtually real-arithemetic types in the next definition is central to our paper.
\begin{definition}
\label{def:vra-types}
An arithmetic type~$\rho$ is called real-arithmetic if its semi-simplification~$\rho^\semisimple$ extends to~$\SL{2}(\RR)$. An arithmetic type~$\rho$ is called virtually real-arithmetic if the restriction~$\Res^\Gamma_{\Gamma'}\,\rho$ to a finite index subgroup~$\Gamma' \subseteq \Gamma$ is real-arithmetic.
\end{definition}
Throughout the paper, we will refer to virtually real-arithmetic types as vra~types. Note that our definition of real-arithmetic types does not require that $\rho$ extends to $\SL{2}(\RR)$, but only~$\rho^\semisimple$.

\begin{example}\label{ex:vra-types}
\begin{enumeratearabic}
\item
Real-arithmetic types of depth~$0$ (cf.\ Section~\ref{ssec:socle-series} for the notion of depth) are direct sums of symmetric powers. Spe\-cif\-ical\-ly, each of them is isomorphic to
\begin{gather*}
  \bigoplus_{\sfd=0}^\infty a_\sfd\, \sym^\sfd
\tx{,}\quad
  a_\sfd \in \ZZ_{\ge 0}
\end{gather*}
for some~$a_\sfd$ such that $\sum a_\sfd < \infty$. In particular, any number of copies of the trivial representation is a real-arithmetic type.

\item
Arithmetic types with finite index kernel are virtually real-arithmetic. Indeed, the restriction to their kernel is isotrivial, i.e.\ a direct sum of trivial types~$\bbone \cong \sym^0$. It thus extends to~$\SL{2}(\RR)$ as required.

\item
Real-arithmetic types of depth~$1$ are extensions of two depth~$0$ representations $\rho$ and $\rho'$:
\begin{gather*}
  \bigoplus_{\sfd=0}^\infty
  a_\sfd\, \sym^\sfd
\lhra
  \rho
\lthra
  \bigoplus_{\sfd=0}^\infty
  a'_\sfd\, \sym^\sfd
\tx{.}
\end{gather*}
Since $\sym^{\vee\,\sfd} \cong \sym^\sfd$, we can use the Clebsch-Gordan rules to compute the possible extensions:
\begin{gather*}
  \Ext^1\big( \rho',\, \rho \big)
\;\cong\;
  \rmH^1\big( \Gamma,\, \rho^{\prime\,\vee} \otimes \rho  \big)
=
  \bigoplus_{\sfd, \sfd' = 0}^\infty
  \bigoplus_{\sfd'' = |\sfd - \sfd'|}^{\sfd + \sfd'}\;
  a_{\sfd} a'_{\sfd'}\,
  \rmH^1\big( \Gamma, \sym^{\sfd''} \big)
\end{gather*}
Parabolic extensions can thus be determined by the Eichler-Shimura theorem.

\item
Finite index induction preserves virtually real-arithmetic types---see Section~\ref{ssec:induction-of-types}. Inclusion~\eqref{eq:inclusion-induction-restriction} allows us to focus on modular forms for the induction of real-arithmetic types, if we wish.

\item
Vector-valued Hecke operators~\cite{raum-2017} map modular forms of vra~type to modular forms of vra~type.
\end{enumeratearabic}
\end{example}

\subsection{Some invariants of vra types}

Let $\rho$ be a vra~type of depth~$0$ and $\Gamma' \subseteq \Gamma$ such that $\Res^\Gamma_{\Gamma'}\, \rho$ is isomorphic to $\bigoplus a_\sfd\, \sym^\sfd$. Then we define the (weight) shift of~$\rho$ as
\begin{gather}
  \shift(\rho)
:=
  \max \big\{ \sfd \,:\, a_\sfd \ne 0 \big\}
\tx{.}
\end{gather}
For general vra~types~$\rho$ of socle length~$j$, we let
\begin{align}
  \shift(\rho)
&:=
  \sum_{i = 1}^{j} \shift\big( \soc^i(\rho) \slash \soc^{i-1}(\rho) \big)
\quad\tx{and}\\
  \pxs(\rho)
&:=
  \#\big\{ 1 \le i \le j \,:\, \soc^{i-1}(\rho) \subset \soc^i(\rho) \;\tx{is not parabolic} \big\}
\end{align}
be the (weight) shift and the parabolic excess of~$\rho$.

\subsection{Universal and co-universal extensions of arithmetic types}
\label{ssec:universal-extension}

Given arithmetic types $\rho$ and $\rho'$, we let $\rho \extb \rho'$ and $\rho \extbpara \rho'$ be the representations that fit into the short exact sequences
\begin{gather*}
  \rho
\lhra
  \rho \extb \rho'
\lthra
  \rho' \otimes \Ext^1(\rho',\, \rho)
\quad\tx{and}\quad
  \rho
\lhra
  \rho \extbpara \rho'
\lthra
  \rho' \otimes \Extpara(\rho',\, \rho)
\end{gather*}
such that for each nonzero $\varphi \in \Ext^1(\rho', \rho)$ or $\varphi \in \Extpara(\rho', \rho)$ we have a direct summand
\begin{gather}
\label{eq:universal-extensions-subtype}
  \rho \boxplus_\varphi \rho' \otimes \varphi
\subseteq
  \rho \extb \rho'
\quad\tx{and}\quad
  \rho \boxplus_\varphi \rho' \otimes \varphi
\subseteq
  \rho \extbpara \rho'
\tx{,}
\end{gather}
respectively. These extensions are called the universal (parabolic) extension of~$\rho$ and~$\rho'$. The co-universal (parabolic) extension is analogously given by
\begin{gather*}
  \rho \otimes \Ext^1(\rho',\, \rho)
\lhra
  \rho \extbd \rho'
\lthra
  \rho'
\quad\tx{and}\quad
  \rho \otimes \Extpara(\rho',\, \rho)
\lhra
  \rho \extbdpara \rho'
\lthra
  \rho'
\tx{.}
\end{gather*}

We clearly have the inclusions
\begin{gather*}
  \rho \extbpara \rho'
\lhra
  \rho \extb \rho'
\quad\tx{and}\quad
  \rho \extbdpara \rho'
\lhra
  \rho \extbd \rho'
\tx{.}
\end{gather*}
We record that if $\rho$ and $\rho'$ are vra~types, then so are $\rho \extb \rho'$ and $\rho \extbd \rho'$.

The projections $\rho \extb \rho' \thra \rho' \otimes \Ext^1(\rho', \rho)$ and $\rho \extbd \rho' \thra \rho'$ yield maps
\begin{gather*}
  \rmM_k \big( \rho \extb \rho' \big)
\lra
  \rmM_k \big( \rho' \otimes \Ext^1(\rho',\, \rho) \big)
\quad\tx{and}\quad
  \rmM_k \big( \rho \extbd \rho' \big)
\lra
  \rmM_k \big( \rho' \big)
\tx{.}
\end{gather*}
In general, they are not surjective (cf.\ Section~\ref{sec:existence-of-modular-forms}). We also obtain maps
\begin{gather*}
  \rmM_k \big( \rho \big)
\lra
  \rmM_k \big( \rho \extb \rho' \big)
\quad\tx{and}\quad
  \rmM_k \big( \rho \otimes \Ext^1(\rho',\, \rho) \big)
\lra
  \rmM_k \big( \rho \extbd \rho' \big)
\end{gather*}
from the inclusions $\rho \hra \rho \extb \rho'$ and $\rho \otimes \Ext^1(\rho',\, \rho) \hra \rho \extbd \rho'$. Less obviously, there is a map from modular forms of co-universal to universal type.
\begin{proposition}
\label{prop:from-co-universal-to-universal}
Fix arithmetic types~$\rho$ and $\rho'$ and a basis $\{ \varphi_i \}$ of\/ $\Ext^1(\rho', \rho)$.  There is a map
\begin{gather*}
  \rho \extbd \rho'
\lra
  \rho \extb \rho'
\tx{,}\quad
  \sum_i v_i \otimes \varphi_i \boxplus v'
\lmto
  \sum_i v_i \boxplus v' \otimes \varphi_i
\tx{.}
\end{gather*}
\end{proposition}
\begin{proof}
Observe that for every decomposition $U_1 \oplus U_2$ of $\Ext^1(\rho', \rho)$ there is a projection
\begin{gather*}
  \rho \extbd \rho'
\lthra
  \big(\, \rho \extbd \rho' \,\big) \big\slash \big( \rho \otimes U_2 \big)
\;\cong\;
  \rho \otimes U_1 \boxplus \rho'
\end{gather*}
When fixing a basis $\{ \varphi_i \}$ of $\Ext^1(\rho',\rho)$, this provides us with a map
\begin{gather*}
  \rho \extbd \rho'
\lra
  \bigoplus_i \rho \otimes \CC \varphi_i \boxplus \rho'
\tx{.}
\end{gather*}
Furthermore, for every $\varphi \in \Ext^1(\rho',\rho)$, we have
\begin{gather*}
  \rho \otimes \CC \varphi \boxplus \rho'
\cong
  \rho \boxplus_\varphi \rho'
\cong
  \rho \boxplus \rho' \otimes \CC \varphi
\tx{.}
\end{gather*}
We obtain the proposition on combining this with the inclusion
\begin{gather*}
  \rho \boxplus \rho' \otimes \CC \varphi
\lhra
  \rho \extb \rho'
\end{gather*}
from~\eqref{eq:universal-extensions-subtype}.
\end{proof}

\begin{proposition}
\label{prop:modular-forms-from-co-universal-to-universal}
Given $k \in \ZZ$, the map
\begin{gather}
\label{eq:prop:modular-forms-from-co-universal-to-universal}
  \rmM_k \big( \rho \extbd \rho' \big)
\lra
  \rmM_k \big( \rho \extb \rho' \big)
\tx{,}\quad
  \sum_i f_i \otimes \varphi_i \boxplus f'
\lmto
  \sum_i f_i \boxplus f' \otimes \varphi_i
\end{gather}
arising from the homomorphism in Proposition~\ref{prop:from-co-universal-to-universal} has kernel
\begin{gather*}
  \rmM_k(\rho)
  \otimes
  \big\{ \sum c_i \varphi_i \in \Ext^1(\rho',\rho) \,:\, \sum c_i = 0 \big\}
\tx{.}
\end{gather*}
\end{proposition}
\begin{proof}
Suppose that $\sum_i f_i \otimes \varphi_i \boxplus f'$ lies in the kernel of~\eqref{eq:prop:modular-forms-from-co-universal-to-universal}. Then we immediately see that $f' = 0$. This implies that $f_i \in \rmM_k(\rho) \subseteq \rmM_k(\rho \extbd \rho')$. The kernel of~\eqref{eq:prop:modular-forms-from-co-universal-to-universal} therefore is
\begin{gather*}
  \big\{ \sum f_i \varphi_i \in \rmM_k(\rho) \otimes \Ext^1(\rho,\rho') \,:\, \sum f_i = 0 \big\}
\tx{.}
\end{gather*}
This implies the desired statement, since $\rmM_k(\rho)$ is finite-dimensional by Proposition~\ref{prop:modular-forms-space-dimension-bound}, which we will prove independently of Proposition~\ref{prop:modular-forms-from-co-universal-to-universal}.
\end{proof}

\subsection{Differential recursions for vector-valued modular forms}
\label{ssec:differential-recursions}

In this section we will prove the foundation to Theorem~\ref{thm:main:recursion-rationality}: A rationality statement for Fourier series coefficients and Taylor coefficients of vector-valued modular forms. For the proof we combine two ingredients: A differential equation that can be deduced for vector-valued modular forms of any type~\cite{mason-2007,knopp-mason-2011}, and a rationality argument that is equally general and can be found, for example, in~\cite{doran-gannon-movasati-shokri-2013}. In particular, Theorem~\ref{thm:recursion-rationality} does not require vra~types, and is certainly known to experts.

For simplicity, we will assume in this section that $\Gamma = \SL{2}(\ZZ)$, implying that $\rmM_\bullet(\rho)$ is free of rank~$\dim\,\rho$ as a module over~$\rmM_\bullet$ by~\cite{marks-mason-2010,knopp-mason-2011}. This assumption can be easily satisfied by passing from $\rho$ for arbitrary~$\Gamma$ to $\Ind_\Gamma^{\SL{2}(\ZZ)}\,\rho$ without impacting the statement of Theorem~\ref{thm:recursion-rationality}. Fix an $\rmM_\bullet$-basis $f_i \in \rmM_{k_i}(\rho)$, $1 \le i \le \dim\,\rho$ of $\rmM_\bullet(\rho)$ with increasing weights~$k_i$.

Recall the Serre derivative on modular forms of weight~$k$:
\begin{gather*}
  \rmD\, f
\;:=\;
  \rmD_k\, f
\;:=\;
  \frac{\partial_\tau}{2 \pi i}\, f
  -
  \frac{k}{12} E_2(\tau)\, f
\tx{,}
\end{gather*}
where
\begin{gather*}
  E_2(\tau)
=
  1 - 24 \sum_{n \ge 1} \sigma_1(n)\, e(n \tau)
\in
  \rmQ\rmM^{[1]}_2
\tx{.}
\end{gather*}
is the quasi-modular Eisenstein series of weight~$2$ and level~$1$. The Serre derivative is covariant with respect to~$\SL{2}(\ZZ)$ from weight~$k$ to weight~$k+2$. In particular, for any arithmetic type~$\rho$, we obtain a holomorphic differential operator
\begin{gather*}
  \rmD_k :\,
  \rmM_k(\rho)
\lra
  \rmM_{k+2}(\rho)
\tx{.}
\end{gather*}
When suppressing the subscript of~$\rmD$ acting on modular forms, we assume that it acts with~$k$ matching the weight. In particular, we use the notation $\rmD^n=\rmD_k^n=\rmD_{k+2n-2}\circ\ldots\rmD_{k+2}\circ\rmD_k$ for the iterated Serre derivative. 

We employ the Serre derivative to derive modular linear differential equations for~$f \in \rmM_k(\rho)$. Let $\rho_f \hra \rho$ be the subrepresentation that is generated by the component functions of~$f$:
\begin{gather*}
  V(\rho_f)
=
  \CC\lspan\big\{ f(\tau) \,:\, \tau \in \HS \big\}
\tx{.}
\end{gather*}
By Equation~(44) of~\cite{knopp-mason-2011}, we arrive at the assertion that $f$ satisfies a differential equation
\begin{gather}
\label{eq:mlde-for-modular-form}
  \sum_{j=0}^{\dim\,\rho_f} g_{f,j}\cdot \rmD^j f
=
  0
\end{gather}
for modular forms~$g_{f,j} \in \rmM_{l - 2j}$, $l \in \ZZ_{\ge 0}$ with nonzero~$g_{f,0}$. We may assume that $l$ is minimal with this property. Then $g_{f,0}$ is unique up to scalar multiples by Theorem~3.14 of~\cite{knopp-mason-2011}. We can rewrite~\eqref{eq:mlde-for-modular-form} by expanding the definition of the Serre deriviative and obtain
\begin{gather}
\label{eq:expandedmlde-for-modular-form}
  \sum_{j=0}^{\dim\,\rho_f} h_{f,j}\cdot \Big(\frac{\partial_\tau}{2 \pi i} \Big)^j f
=
  0
\end{gather}
for quasimodular forms~$h_{f,j}$ of level~$1$ and weight~$l - 2j$.

\begin{theorem}
\label{thm:recursion-rationality}
Fix $k \in \ZZ$ and an arithmetic type~$\rho$. Given $f \in \rmM_k(\rho)$, the following are finite-dimensional as $\QQ$ vector spaces:
\begin{alignat*}{2}
&
  \lspan \QQ\big\{
  c(f; n; \tau) \,:\,
  n \in \QQ_{\ge 0}
  \big\}
&&\subset
  V(\rho) \otimes \Poly(\tau)
\qquad\tx{and}
\\&
  \lspan \QQ\big\{
  c_{\tau_0}(f; n) \,:\,
  n \in \QQ_{\ge 0}
  \big\}
&&\subset
  V(\rho)
\tx{.}
\end{alignat*}

More precisely, there are quasimodular forms~$h_{f,j}$, $0 \le j \le \dim\,\rho_f$ of level~$1$ and weight $l - 2j$, $l \in \ZZ$ such that the Fourier coefficients of~$f$ satisfy the differential recursion
\begin{gather*}
  \sum_{\substack{i,j \in \ZZ_{\ge 0} \\ i+j \le \dim\,\rho_f}}
  \sum_{\substack{m \in \ZZ_{\ge 0} \\ m \le n}}
  c(h_{f,i+j}; n - m)
  m^i\, \partial_\tau^j\, c(f; m; \tau)
=
  0
\end{gather*}
for all $n \in \QQ_{\ge 0}$, and the Taylor coefficients of~$f$ at~$\tau_0$ satisfy the recursion
\begin{gather*}
  \sum_{m = 0}^n
  \sum_{j = 0}^{\dim\,\rho_f}
  \tbinom{n}{m}
  c_{\tau_0}(h_{f,j}; n - m)
  c_{\tau_0}(f; m)
=
  0
\end{gather*}
for all $n \in \ZZ_{\ge 0}$.
\end{theorem}
\begin{remark}
In the case of vra~types the recursion in~Theorem~\ref{thm:recursion-rationality} is effectively computable. This can be employed to compute Fourier or Taylor expansions up to order~$n$ of modular forms of vra~type in $\mathrm{TIME}(n^2\log(n))$. 

Instead of the Taylor expansion of a modular form~$f$ at $\tau_0$ one may equally well consider the Taylor expansion of $f \circ g$ at~$0$ for a transformation $g \in \SL{2}(\CC)$ that maps~$\HS$ to the Poincar\'e disc and~$\tau_0$ to~$0$. If $g \in \SL{2}(\QQ(\tau_0))$ and $\tau_0$ is algebraic, then the Taylor coefficients of~$f \circ g$ at~$0$ satisfy the same rationality statement as in Theorem~\ref{thm:recursion-rationality}.
\end{remark}
\begin{proof}[{Proof of Theorem~\ref{thm:recursion-rationality}}]
The recursions hold by~\eqref{eq:expandedmlde-for-modular-form}. To deduce the first statement from this, it suffices to show that $c(h_{f,j};0) \ne 0$ for some~$j$ and that $c_{\tau_0}(h_{f,j};\,0) \ne 0$ for some $j$. The former has literally been established by Knopp and Mason in their proof of Theorem~3.14 of~\cite{knopp-mason-2011}. The latter follows along the same lines when considering that for every~$\tau_0 \in \HS$ there is a modular form of level~$1$ and weight at most~$12$ that vanishes at $\SL{2}(\ZZ)\tau$ to order~$1$ and nowhere else.
\end{proof}

\section{Connection to other notions of modular forms}
\label{sec:other-notions-of-modular-forms}

In this section, we show that several classical constructions can be connected to modular forms of virtually real-arithmetic type. Classical modular forms appear for all symmetric powers $\sym^\sfd$. Mock modular forms yield modular forms for real-arithmetic types of depth~$1$. Higher order modular forms yield modular forms of virtually real-arithmetic types whose semi-simplification is isotrivial. Iterated integrals yield modular forms for virtually real-arith\-metic types whose socle extension classes are connected to iterated Eichler-Shimura integrals themselves. None of these constructions, however, exhausts modular forms of virtually real-arithmetic type. 

\subsection{Classical modular forms and arithmetic types~\texpdf{$\sym^\sfd$}{sym(d)}}
\label{ssec:other-notions-of-modular-forms:classical}

The goal of this section is to reformulate and extend results of~\cite{kuga-shimura-1960} on the spaces $\rmM_k(\sym^\sfd)$. The main result is recorded in Proposition~\ref{prop:embedding-of-modular-forms}. We base the whole section on two observations: The covariance of the vector-valued raising operator in Equation~\eqref{eq:def:raising-operator-vector-valued} and the $\SL{2}(\RR)$-invariance of the vector-valued polynomial in Equation~\eqref{eq:invariant-polynomial-negative-weight}.

We start by defining a vector-valued raising operator. We could extract it in essence from either~\cite{kuga-shimura-1960} or~\cite{zemel-2015}, but prefer to give a clean and short proof of its covariance. Recall that $\Poly(X,1)$ with the slash action of weight~$-1$ is isomorphic to $\std = \sym^1$ as an~$\SL{2}(\RR)$-re\-pres\-en\-ta\-tion. Set
\begin{gather}
\label{eq:def:raising-operator-vector-valued}
  \vvR_k 
:=
  (X - \tau) \partial_\tau - k 
  :\,
  \cC^\infty_k(\rho)
\lra
  \cC^\infty_{k+1}(\std \otimes \rho)
\tx{,}
\end{gather}
Throughout the paper, we use the notation $\vvR^\sfd_k := \vvR_{k + \sfd-1} \circ \cdots \circ \vvR_k$.

It is possible to give a closed formula for the action of $\vvR^\sfd_k$:
\begin{lemma}
\label{la:raising-operator-power-formula}
For any smooth $f$ we have that
\begin{align*}
  \vvR^\sfd_k\, f
=
  \sum_{j=0}^\sfd
  \binom{\sfd}{j}\,
  (k+\sfd-1)_{\sfd-j}\,
  (-1)^{\sfd-j}
  (X - \tau)^j\,
  \partial_\tau^j\,
  f
\tx{,}
\end{align*}
where $(a)_n := a(a-1)\ldots(a-n+1)$ denotes the falling factorial.
\end{lemma}
\begin{proof}
This can be verified by induction on~$\sfd$.
\end{proof}

Observe that $\vvR_k$ is a derivation on the graded modules~$\cC^\infty_\bullet(\rho) = \bigoplus \cC^\infty_k(\rho)$. That is, we have
\begin{gather}
\label{eq:raising-operator-derivation}
  \vvR_{k + l}\, (f \otimes g)
=
  ( \vvR_{k}\, f ) \otimes g
  +
  f \otimes ( \vvR_{l}\, g )
\end{gather}
for $f \in \cC^\infty_{k}(\rho_f)$ and $g \in \cC^\infty_{l}(\rho_g)$.

The following covariance property of~$\vvR_k$ confirms that it is a (weight) raising operator. As opposed to all other raising operators for elliptic modular forms that the authors are aware of, it intertwines two different arithmetic types.
\begin{proposition}
\label{prop:raising-operator-covariant}
Let $\rho$ be an arithmetic type and $k \in \ZZ$. The map $\vvR_k$ is covariant with respect to $\SL{2}(\ZZ)$ from $\cC^\infty_k(\rho)$ to $\cC^\infty_{k+1}(\std \otimes \rho)$. If $\rho$ extends to an $\SL{2}(\RR)$-representation, it is covariant with respect to $\SL{2}(\RR)$.
\end{proposition}
\begin{proof}
We check covariance with respect to generators $\begin{psmatrix} 1 & b \\ 0 & 1 \end{psmatrix}$ and $S$ of $\SL{2}(\ZZ)$ (i.e.\@ $b \in \ZZ$) and $\SL{2}(\RR)$ (i.e.\@ $b \in \RR$). The first case is
\begin{multline*}
  \big( \vvR_k f \big)
  \big|_{k+1,\std \otimes \rho}\,\begin{psmatrix} 1 & b \\ 0 & 1 \end{psmatrix}
=
  \Big( \big( (X - \tau) \partial_\tau - k \big)\, f \Big)
  \big|_{k+1,\std \otimes \rho}\, \begin{psmatrix} 1 & b \\ 0 & 1 \end{psmatrix}
\\
=
  \big( ((X + b) - (\tau + b)) \partial_\tau - k \big)\,
  \big( f \big|_{k,\rho}\, \begin{psmatrix} 1 & b \\ 0 & 1 \end{psmatrix} \big)
=
  \vvR_k\, \big( f \big|_{k,\rho}\,\begin{psmatrix} 1 & b \\ 0 & 1 \end{psmatrix} \big) 
\tx{.}
\end{multline*}
Covariance with respect to~$S$ follows along the same lines when using
\begin{gather*}
  \partial_\tau \Big( \tau^{-k} f \big( \tfrac{-1}{\tau} \big) \Big)
=
  -k \tau^{-k-1} f \big( \tfrac{-1}{\tau} \big)
  +
  \tau^{-k-2} (\partial_\tau f) \big( \tfrac{-1}{\tau} \big)
\tx{.}
\end{gather*}
Indeed, we have
\begin{align*}
&\hphantom{={}}
  \big( \vvR_k f \big)
  \big|_{k+1, \std \otimes \rho}\, S
\\
&=
  \Big( \big( (X - \tau) \partial_\tau - k \big)\, f \Big)
  \big|_{k+1, \std \otimes \rho}\, S
\\
&=
  \rho(-S)\;
  \tau^{-k-1}\,
  \Big( \big( (- 1 - \tau X) \partial_\tau - k X \big)\, f \Big) \big( \tfrac{-1}{\tau} \big)
\\
&=
  \rho(-S)\;
  \Big(
  \big( - 1 - \tfrac{-1}{\tau} X \big)\, \tau^{-k-1} (\partial_\tau f) \big( \tfrac{-1}{\tau} \big)
  -
  k X \tau^{-k-1} f \big( \tfrac{-1}{\tau} \big)
  \Big)
\\
&=
  \rho(-S)\;
  \Big(
  \big( - 1 - \tfrac{-1}{\tau} X \big)\, \tau
  \big( \partial_\tau (f \big|_k\, S) + k \tau^{-1} (f \big|_k\, S) \big)
  -
  k X \tau^{-1} (f \big|_k\, S)
  \Big)
=
  \vvR_k\big( f_{k,\rho}\, S \big)
\tx{.}
\end{align*}
\end{proof}

We let
\begin{gather*}
  \pi_{\std^\sfd}
:\,
  \std^\sfd \lra \sym^\sfd
\tx{,}\;
  p(X_1, \ldots, X_\sfd)
\lmto
  p(X, \ldots, X)
\end{gather*}
be the symmetrization map. Iterating the raising operator~$\vvR$ and then applying~$\pi_{\std^\sfd}$, we obtain covariant maps
\begin{gather}
\label{eq:def:higher-vector-valued-raising-operator}
  \pi_{\std^\sfd} \circ \vvR_k^\sfd
:\,
  \cC^\infty_k(\rho) \lra \cC^\infty_{k+\sfd}(\sym^\sfd \otimes \rho)
\tx{.}
\end{gather}

\begin{lemma}
\label{la:raising-operator-almost-injective}
If $k \ne 0$ then $\vvR_k$ is injective. Specifically, it has a left inverse
\begin{gather*}
  \frac{(X - \tau) \partial_X - 1}{k}
\tx{.}
\end{gather*}

Suppose that $k > 0$. Then
\begin{gather}
\label{eq:prop:injective-vector-valued-raising-operator}
  \pi_{\std^\sfd} \circ \vvR_k^\sfd
:\,
  \cC^\infty_k(\rho) \lra \cC^\infty_{k+\sfd}(\sym^\sfd \otimes \rho)
\end{gather}
is injective on polynomials.
\end{lemma}
\begin{proof}
The first part is an easy computation. The second part follows when considering the constant coefficient of~\eqref{eq:prop:injective-vector-valued-raising-operator} with respect to~$X$. It equals
\begin{gather*}
  (-1)^\sfd\,
  (\tau \partial_\tau + k + \sfd - 1)
  \cdots
  (\tau \partial_\tau + k)
\tx{.}
\end{gather*}
\end{proof}

Observe that for integers $\kappa \ge 0$ the polynomial
\begin{gather}
\label{eq:invariant-polynomial-negative-weight}
  (Y - \tau)^\kappa
\in
  \cC^\infty_{-\kappa}\big(\Poly(Y,\kappa) \big)
\end{gather}
is invariant with respect to~$\SL{2}(\RR)$. The polynomial in Equation~\eqref{eq:invariant-polynomial-negative-weight} in conjunction with the raising operator in~\eqref{eq:def:higher-vector-valued-raising-operator} allows us to obtain $\SL{2}(\RR)$-covariant operators
\begin{gather}
\label{eq:def:embedding-of-modular-forms}
  p_\kappa \vvR^\sfd_k
:\,
  \cC^\infty_k(\rho)
\lra
  \cC^\infty_{k + \sfd - \kappa}\big(\sym^{\sfd+\kappa} \otimes \rho \big)
\tx{,}\quad
  f
\lmto
  (X - \tau)^\kappa\, \vvR_k^\sfd\, f
\tx{.}
\end{gather}
Note that here $p$ is part of the symbol $p_\bullet \vvR^\bullet_\bullet$, and we explicitly do not let $p$ be the Polynomial~\eqref{eq:invariant-polynomial-negative-weight}.

In analogy with the results in the work of Kuga-Shimura~\cite{kuga-shimura-1960} and Zemel~\cite{zemel-2015}, we obtain the next result.
\begin{proposition}
\label{prop:embedding-of-modular-forms}
Suppose that $k + \sfd$ is odd, $k > \sfd$, or $k < -\sfd$. Then for any arithmetic type~$\rho$ with finite index kernel, the following map is an isomorphism:
\begin{gather}
\label{eq:embedding-of-modular-forms}
\begin{aligned}
  \bigoplus_{\substack{j = -\sfd \\ j \equiv \sfd \pmod{2}}}^\sfd
  \rmM_{k + j}(\rho)
&\lra
  \rmM_k(\sym^\sfd \otimes \rho)
\\
  \big( f_{-\sfd}, \ldots, f_\sfd \big)
&\lmto
  \sum_{\substack{j = -\sfd \\ j \equiv \sfd \pmod{2}}}^\sfd
  p_{\frac{\sfd + j}{2}} \vvR^{\frac{\sfd-j}{2}}_{k + j}\, f_j
\tx{.}
\end{aligned}
\end{gather}
\end{proposition}
\begin{proof}
We view elements of $\rmM_k(\sym^\sfd \otimes \rho)$ as polynomials in~$X - \tau$. To show that the map~\eqref{eq:embedding-of-modular-forms} is injective, it evidently suffices to show that the constant term (with respect to $X - \tau$) of $\vvR^{(\sfd-j) \slash 2}_{k+j} f$ is nonzero for any $f \in \rmM_{k+j}(\rho)$. The defining formula~\eqref{eq:def:raising-operator-vector-valued} in conjunction with the assumptions on $k$, already implies this.

To obtain surjectivity of~\eqref{eq:embedding-of-modular-forms}, we employ the same argument as Kuga-Shimura in~\cite{kuga-shimura-1960}: The lowest nonvanishing term (with respect to $X - \tau$) of any modular form in~$\rmM_k(\sym^\sfd \otimes \rho)$ yields a modular form of type~$\rho$.
\end{proof}

Based on Proposition~\ref{prop:embedding-of-modular-forms}, we can now prove that spaces of modular forms of vra~type are finite-dimensional.
\begin{proposition}
\label{prop:modular-forms-space-dimension-bound}
Let~$k \in \ZZ$ and let $\rho$ be a virtually real-arithmetic type. Then
\begin{gather*}
  \dim\,\rmM_k(\rho)
<
  \infty
\tx{.}
\end{gather*}
\end{proposition}
\begin{proof}
Passing to a sufficiently small subgroup~$\Gamma \subseteq \SL{2}(\ZZ)$, we can assume that $\rho$ is real-arithmetic. If $\rho$ has depth~$0$, it then suffices to prove that $\dim\,\rmM_k(\sym^\sfd) < \infty$ for all integers~$\sfd \ge 0$. To this end, we can use the same argument as at the end of the proof of Proposition~\ref{prop:embedding-of-modular-forms}: The lowest nonvanishing term of any modular in $\rmM_k(\sym^\sfd)$ is a classical modular form. Using induction, we obtain a bound
\begin{gather*}
  \dim\,\rmM_k(\sym^\sfd)
\le
  \sum_{\substack{j = -\sfd \\ j \equiv \sfd \pmod{2}}}^\sfd
  \dim\,\rmM_{k+j}
<
  \infty
\tx{.}
\end{gather*}

We next induce on the depth of~$\rho$. Given a general real-arithmetic type~$\rho$ of depth~$\sfd$, we have an exact sequence
\begin{gather*}
  \soc^\sfd(\rho)
\lhra
  \rho
\lthra
  \rho \slash \soc^\sfd(\rho)
\tx{.}
\end{gather*}
Observe that $\soc^\sfd(\rho)$ is real-arithmetic of depth~$\sfd-1$ and that the quotient of~$\rho$ by $\soc^\sfd(\rho)$ is real-arithmetic of depth~$0$. Employing the induction hypothesis, we find that
\begin{gather*}
  \dim\,\rmM_k\big( \soc^\sfd(\rho) \big)
<
  \infty
\quad\tx{and}\quad
  \dim\,\rmM_k\big( \rho \slash \soc^\sfd(\rho) \big)
<
  \infty
\tx{.}
\end{gather*}
We have an exact sequence (the right arrow is not necessarily a surjection: cf.\@ Theorem~\ref{thm:existence-of-modular-forms})
\begin{gather*}
  \rmM_k\big( \soc^\sfd(\rho) \big)
\lhra
  \rmM_k(\rho)
\lra
  \rmM_k\big( \rho \slash \soc^\sfd(\rho) \big)
\tx{.}
\end{gather*}
Combining it with the previous dimension bounds we finish the proof.
\end{proof}

A coarser variant of Proposition~\ref{prop:embedding-of-modular-forms} that connects $\rmM_k(\sym^\sfd \otimes \rho)$ with almost holomorphic modular forms (and hence quasimodular forms) can be deduced from the operator:
\begin{gather}
\label{eq:def:almost-symmetric-covariant-operator}
  \rmW :\,
  \cH_k(\rho)\big[ y^{-1} \big]
\lra
  \cH_{k-1}(\std \otimes \rho)\big[ y^{-1} \big]
\tx{,}\quad
  f_0(\tau) + y^{-1} f_1(\tau)
\lmto
  (X - \tau) f_0(\tau)- 2 i f_1(\tau)
\tx{,}
\end{gather}
where $f_0$ is holomorphic and $f_1$ is almost holomorphic. By iterating $\rmW$ and symmetrizing, we obtain
\begin{multline}
  \pi_{\std^\sfd} \circ \rmW^\sfd :\,
  \cH_k(\rho) + \cdots + y^{-\sfd}\cH_k(\rho)
\lra
  \cH_{k-1}(\sym^\sfd\otimes \rho)
\tx{,}
\\
  f_0(\tau) + \cdots + y^{\sfd} f_\sfd(\tau)
\lmto
  (X - \tau)^\sfd f_0(\tau)
  -
  2 i (X - \tau)^{\sfd-1} f_1(\tau)
  +
  \cdots
  +
  (-2i)^\sfd f_\sfd(\tau)
\tx{.}
\end{multline}
It is clear that $\rmW$ is invertible, and so Lemma~\ref{la:almost-symmetric-covariance} establishes a correspondence of almost holomorphic forms of depth at most~$\sfd$ and modular forms of type~$\sym^\sfd$. Observe that~\cite{zemel-2015} rests on this idea, but~$\rmW$ is not spelled out in the same way as here.
\begin{lemma}
\label{la:almost-symmetric-covariance}
The operator~$\rmW$ in~\eqref{eq:def:almost-symmetric-covariant-operator} is covariant with respect to~$\SL{2}(\ZZ)$. If $\rho$ extends to an $\SL{2}(\RR)$-representation, it is covariant with respect to $\SL{2}(\RR)$.
\end{lemma}
\begin{proof}
It is clearly invariant with respect to the action of~$\begin{psmatrix} 1 & b \\ 0 & 1 \end{psmatrix}$ with $b \in \ZZ$ or $b \in \RR$, respectively. Invariance with respect to $S$ follows by a direct computation (using that $y^{-1} \circ S = y^{-1} \tau (\tau - 2 i y)$):
\begin{align*}
  \big( f_0 + y^{-1} f_1 \big) \big|_k\,S
&\;=\;
  \big( f_0\big|_k\,S - 2 i f_1 \big|_{k-1}\,S \big)
  \,+\,
  y^{-1} f_1 \big|_{k-2}\,S
\tx{,}
\\
  \big( (X - \tau) f_0 - 2 i f_1 \big) \big|_{k-1}\,S
&\;=\;
  (X - \tau) \big(f_0 \big|_k\,S - 2 i f_1 \big|_{k-1}\,S \big)
  \,-\,
  2 i f_1 \big|_{k-2}\,S
\tx{.}
\end{align*}
\end{proof}

\subsection{Mixed mock modular forms}
\label{ssec:other-notions-of-modular-forms:mixed-mock-modular-forms}

The goal of this section is to show that mixed mock modular forms are components of specific modular forms of virtually real-arithmetic type. Indeed, we recognize mixed mock modular forms as components of modular forms of vra~type that have depth as most~$1$.

Recall the universal and co-universal parabolic extensions from Section~\ref{ssec:universal-extension}.
\begin{theorem}
\label{thm:mock-modular-as-co-universal-type-modular-form}
Given $k \in \ZZ_{\le 0}$, an arithmetic type~$\rho$ with finite index kernel, and a sum decomposition $|k| = \sfd + \sfd'$ with $\sfd,\sfd' \in \ZZ_{\ge 0}$, there is an inclusion
\begin{gather}
\label{eq:thm:mock-modular-as-co-universal-type-modular-form}
  \bbM_k(\rho)
\lhra
  \rmM^!_{k + \sfd}\big( \rho \otimes \sym^\sfd \,\extbdpara\, \sym^{\sfd'} \big)
\tx{,}\quad
  f
\lmto
  \big( (\sfd + \sfd')_\sfd^{-1}\, \vvR^\sfd_k\, f \big) \otimes \varphi_f
  \,\boxplus\,
  (-1) (X - \tau)^{\sfd'}
\tx{,}
\end{gather}
where $\varphi_f \in \Hpara(\Gamma, \sym^{|k|})$ is mapped into $\Hpara(\Gamma, \sym^\sfd \otimes \sym^{\vee\,\sfd'})$ via the embedding in~\eqref{eq:def:sym-power-product-embedding} and the pairing in~\eqref{eq:def:sym-power-pairing}.
\end{theorem}
\begin{remark}
The most straightforward instance of Theorem~\ref{thm:mock-modular-as-co-universal-type-modular-form} that does not require the use of~\eqref{eq:def:sym-power-product-embedding} and~\eqref{eq:def:sym-power-pairing} is the case of~$\sfd = |k|$ and~$\sfd' = 0$.
\end{remark}

Combining the statement of Theorem~\ref{thm:mock-modular-as-co-universal-type-modular-form} with Proposition~\ref{prop:from-co-universal-to-universal}, we obtain embeddings into spaces of modular forms of universal type.
\begin{corollary}
\label{cor:mock-modular-as-universal-type-modular-form}
Given $k \in \ZZ_{\le 0}$, an arithmetic types~$\rho$ with finite index kernel, and a sum decomposition $|k| = \sfd + \sfd'$ with $\sfd,\sfd' \in \ZZ_{\ge 0}$, there is an inclusion
\begin{gather}
\label{eq:cor:mock-modular-as-universal-type-modular-form}
  \bbM_k(\rho)
\lhra
  \rmM^!_{k + \sfd}\big( \rho \otimes \sym^\sfd \,\extbpara\, \sym^{\sfd'} \big)
\tx{,}\quad
  f
\lmto
  \big( (\sfd + \sfd')_\sfd^{-1}\, \vvR^\sfd_k\, f \big) \,\boxplus\, (-1) (X - \tau)^{\sfd'} \otimes \varphi_f
\tx{.}
\end{gather}
\end{corollary}
\begin{example}\label{ex:vra-type-delta}
The first case of a mock modular form~$f_\Delta$ in level~$1$ that is not an Eisenstein series can be observed in weight~$-10$. Its shadow is a scalar multiple of Ramanujan's $\Delta$-function. We can extract the cocycle~$\varphi=\varphi_{f_\Delta}$ attached to~$f_\Delta$ from the introduction of~\cite{kohnen-zagier-1984}. It is a linear combination $\varphi = -i \omega_+ \varphi_+ + \omega_- \varphi_-$ with $\omega_+ \approx 0.021446$ and~$\omega_- \approx 0.000048$, and
\begin{align*}
  \varphi_+(S)
&=
  \frac{192}{691}
  -
  \frac{16}{3} X^2
  +
  16 X^4
  -
  16 X^6
  +
  \frac{16}{3} X^8
  -
  \frac{192}{691} X^{10}
\tx{,}
\\
  \varphi_-(S)
&=
  768 X
  -
  4800 X^3
  +
  8064 X^5
  -
  4800 X^7
  +
  768 X^9
\tx{.}
\end{align*}
Observe that every cocycle is a linear combination of $\varphi_\pm$ and the cohomologically trivial one given by $\varphi_0(S) = 1 - X^{10}$. The universal extension~$\bbone \,\extbpara\, \sym^{10}$ can be realized by the matrix representation with usual $\rho(T) = \diag(1, \sym^{10}(T), \sym^{10}(T))$ and
\begin{gather*}
  \rho(S)
=
  \begin{pmatrix}
  1 &
  \begin{smallmatrix}
  \frac{192}{691} & 0 & -\frac{16}{3} & 0 & 16 & 0 & -16 & 0 & \frac{16}{3} & 0 & -\frac{192}{691} &
  \end{smallmatrix} &
  \begin{smallmatrix}
  0 & 768 & 0 & -4800 & 0 & 8064 & 0 & -4800 & 0 & 768 & 0
  \end{smallmatrix}
  \\
  & \sym^{10}(S) & \\
  &  & \sym^{10}(S) \\
  \end{pmatrix}
\tx{.}
\end{gather*}

We now obtain $f_\Delta$ as a component of
\begin{multline*}
  \rT \begin{pmatrix}
  f_\Delta & i \omega_+ (X - \tau)^{10} & -\omega_- (X - \tau)^{10}
  \end{pmatrix}
\tx{,}\quad
\tx{where}
\\
  (X - \tau)^{10}
\;\wht{=}\;
  \begin{psmatrix}
  1 & -10 \tau & 45 \tau^2 & -120 \tau^3 & 210 \tau^4 & -252 \tau^5 & 210 \tau^6 & -120 \tau^7 & 45 \tau^8 & -10 \tau^9 & \tau^{10}
  \end{psmatrix}
\tx{.}
\end{multline*}
\end{example}

The range of~\eqref{eq:cor:mock-modular-as-universal-type-modular-form} consists of holomorphic modular forms. It is therefore compatible in a natural way with the tensor product that appears in the definition of mixed mock modular forms (cf.\ Section~\ref{ssec:mixed-mock-modular-forms}). We thus obtain the next statement:
\begin{corollary}
\label{cor:mixed-mock-modular-as-universal-type-modular-form}
Given $k \in \ZZ_{\le 0}$, $l \in \ZZ$, an arithmetic types~$\rho$ with finite index kernel, and a sum decomposition $|k| = \sfd + \sfd'$ with $\sfd,\sfd' \in \ZZ_{\ge 0}$, there are inclusions
\begin{gather}
\label{eq:cor:mixed-mock-modular-as-couniversal-type-modular-form}
\begin{aligned}
  \rmM_l(\rho') \otimes \bbM_k(\rho)
&\lhra
  \rmM^!_{k+\sfd+l}\big( \rho' \otimes \rho \otimes \sym^\sfd \,\extbdpara\, \rho' \otimes \sym^{\sfd'} \big)
\tx{,}\;
\\
  g \otimes f
&\lmto
  g \otimes \big((\sfd + \sfd')_\sfd^{-1}\,  \vvR^\sfd_k\, f \big) \otimes \varphi_f \,\boxplus\, g \otimes (-1) (X - \tau)^{\sfd'}
\tx{,}
\end{aligned}
\end{gather}
and
\begin{gather}
\label{eq:cor:mixed-mock-modular-as-universal-type-modular-form}
\begin{aligned}
  \rmM_l(\rho') \otimes \bbM_k(\rho)
&\lhra
  \rmM^!_{k+\sfd+l}\big( \rho' \otimes \rho \otimes \sym^\sfd \,\extbpara\, \rho' \otimes \sym^{\sfd'} \big)
\tx{,}\;
\\
  g \otimes f
&\lmto
  g \otimes \big((\sfd + \sfd')_\sfd^{-1}\, \vvR^\sfd_k\, f \big) \,\boxplus\, g \otimes (-1) (X - \tau)^{\sfd'} \otimes \varphi_f
\tx{.}
\end{aligned}
\end{gather}
\end{corollary}
\begin{remark}
For simplicity, we have restricted our statements to $l \in \ZZ$. The statement holds for $l \in \frac{1}{2}\ZZ$ if we consider representations $\rho'$ of the metaplectic cover~$\Mp{2}(\ZZ)$ of $\SL{2}(\ZZ)$. Observe however that passage from $k \in \ZZ_{\le 0}$ to $k \in \frac{1}{2}\ZZ$ requires infinite-dimensional representations of~$\Mp{2}(\ZZ)$, and therefore incurs specific issues concerning convergence and approximate units.
\end{remark}

The proof of Theorem~\ref{thm:mock-modular-as-co-universal-type-modular-form} occupies the remainder of this section. It requires the next fact about the raising operator of Section~\ref{ssec:other-notions-of-modular-forms:classical}.
\begin{lemma}
\label{la:raising-operator-converts}
Fix $\sfd, \sfd' \in \ZZ_{\ge 0}$. Given $p \in \Poly(\tau, \sfd + \sfd')$, then
\begin{gather*}
  \vvR^\sfd_{-\sfd-\sfd'}\, p
\in
  \Poly(X, \sfd)
  \otimes
  \Poly(\tau, \sfd')
\end{gather*}
equals the image of $p$ under the map~\eqref{eq:def:sym-power-product-embedding} up to a constant factor:
\begin{gather*}
  \tau^n
\lmto
  (\sfd+\sfd')_\sfd
  \sum_{i+j = n}
  \tbinom{\sfd}{i}
  \tbinom{\sfd'}{j}
  \tbinom{\sfd+\sfd'}{n}^{-1}\,
  X^i \tau^j
\tx{,}\quad
  0 \le i \le \sfd
\tx{,}\;
  0 \le j \le \sfd'
\tx{,}
\end{gather*}
where $(a)_n = a(a-1)\ldots(a-n+1)$ denotes the falling factorial.
\end{lemma}
\begin{proof}
Since every polynomial $p(\tau)$ can be factored completely over~$\CC$, we can view~$p$ as an element of the $(\sfd+\sfd')$\thdash\ tensor power of $\Poly(\tau,1)$. This implies that $\vvR^\sfd_{-\sfd-\sfd'}\, p \in \Poly(X, \sfd) \otimes \Poly(\tau, \sfd')$ when using~\eqref{eq:raising-operator-derivation}. Note that we could have used Lemma~\ref{la:raising-operator-power-formula} together with a short computation, instead.

By the Clebsch-Gordan rules there is a unique copy of $\sym^{\sfd+\sfd'}$ in the tensor product of $\Poly(X, \sfd)$ and $\Poly(\tau, \sfd')$. In conjunction with the covariance of~$\vvR$, it suffices to check the image of~$1 \in \Poly(\tau, \sfd+\sfd')$ under $\vvR^\sfd_{-\sfd-\sfd'}$ using Lemma~\ref{la:raising-operator-power-formula}. It equals~$(-\sfd'-1)_\sfd (-1)^{\sfd} = (\sfd+\sfd')_\sfd$, matching the image of~$1$ under~\eqref{eq:def:sym-power-product-embedding}.
\end{proof}

\begin{proof}[{Proof of Theorem~\ref{thm:mock-modular-as-co-universal-type-modular-form}}]
Fixing $f \in \bbM_k(\rho)$, we have to show that
\begin{gather*}
  \big( (\sfd + \sfd')_\sfd^{-1}\, \vvR^\sfd_k\, f \big) \,\boxplus\, (-1) (X - \tau)^{\sfd'}
\in
  \rmM^!_{k + \sfd}\big( \rho \otimes \sym^\sfd \,\boxplus_{\varphi_f}\, \sym^{\sfd'} \big)
\tx{.}
\end{gather*}
It is clear that $\vvR^\sfd_k\, f$ is holomorphic and has at most exponential growth.

Using $k + \sfd = -\sfd'$, it follows from Section~\ref{ssec:symmetric-powers} that
\begin{gather*}
  (X - \tau)^{\sfd'}
\in
  \rmM_{k+\sfd}\big( \sym^{\sfd'} \big)
\tx{.}
\end{gather*}
It remains to check the transformation behavior of the first component. For this, we have to examine how the isomorphism between $\Poly(X,\sfd')$ and its dual interacts with the raising operator. By definition of $\varphi_f$ in Section~\ref{ssec:cocycles-attached-to-functions} and the corvariance of~$\vvR$ from Proposition~\ref{prop:raising-operator-covariant}, we have
\begin{gather*}
  \big( \vvR^\sfd\, f \big)
  \big|_{k+\sfd,\sym^\sfd \otimes \rho}\, \gamma
=
  \vvR^\sfd\big(
  f \big|_{k,\rho}\, \gamma
  \big)
=
  \vvR^\sfd\big(
    f + \varphi_f(\gamma) 
  \big)
=
  \vvR^\sfd\, f
  \,+\,
  \vvR^\sfd\, \varphi_f(\gamma)
\tx{,}\quad
  \gamma \in \Gamma
\tx{.}
\end{gather*}
The contribution of the second component to the first one, when $\gamma$ acts, is given by
\begin{gather*}
  -
  \big\langle \varphi_f(\gamma),\,
  (X - \tau)^{\sfd'}
  \big\rangle
\tx{.}
\end{gather*}
The proof is hence reduced to showing that
\begin{gather*}
  (\sfd + \sfd')_\sfd^{-1}\,
  \vvR^\sfd\, \varphi_f(\gamma)
=
  \big\langle \varphi_f(\gamma),\,
  (X - \tau)^{\sfd'}
  \big\rangle
\tx{.}
\end{gather*}
On the right-hand side, we view $\varphi_f(\gamma) \in \Poly(\tau, \sfd+\sfd')$ as an element of~$\Poly(\tau, \sfd) \otimes \Poly(\tau, \sfd')$ by the Embedding~\eqref{eq:def:sym-power-product-embedding}. Lemma~\ref{la:raising-operator-converts}, which computes the left-hand side, in conjunction with~\eqref{eq:sym-power-pairing-variable-substitution} confirms the previous equation. This concludes the proof.
\end{proof}

\subsection{Higher order modular forms}
\label{ssec:higher-order-modular-forms}

The definition of second order modular forms goes back to Goldfeld~\cite{goldfeld-1999}, who considered Eisenstein series twisted by modular symbols. The notion was systematized by Chinta, Diamantis, and O'Sullivan in~\cite{chinta-diamantis-osullivan-2002}, which includes the definition of higher order modular forms. Some of the ideas in this section can also be found in~\cite{jorgenson-osullivan-2008}; However, they were not pursued in a systematic manner.

Observe that higher order modular forms for the full modular group are modular (i.e.\ modular of order~$0$ in the terminology of higher order modular forms), since $\Hpara(\SL{2}(\ZZ), \bbone) = \{0\}$. For this reason, we need the freedom to let $\Gamma$ be any finite index subgroup of $\SL{2}(\ZZ)$. In particular, $\Gamma$ may be any congruence subgroup of $\SL{2}(\ZZ)$.

The goal of this section is to show that higher order modular forms are instances of modular forms for real-arithmetic types. When combining this statement with induction of modular forms as explained in~\cite{raum-2017}, then higher order modular forms can be recognized as components of modular forms of virtually real-arithmetic types whose socle factors are virtually isotrivial.

We start by giving the (recursive) definition of higher order modular forms, which has its origins in~\cite{diamantis-sreekantan-2006,chinta-diamantis-osullivan-2002}.
\begin{definition}
\label{def:higher-order-modular-forms}
Modular forms of order~$\sfd = 0$ are defined as modular forms in the usual sense. We call $f \in \cH_k^\md$ a modular form of order $\sfd > 0$ and of weight~$k \in \ZZ$ if for all $\gamma \in \Gamma$ the function $f \big|_k\, (\gamma - 1)$ is a modular form of order~$\sfd-1$ and weight~$k$.
\end{definition}
We denote the space of order~$\sfd$ modular forms that have weight~$k$ by $\rmM^{[\sfd]}_k = \rmM^{[\sfd]}_k(\Gamma)$.

For the purpose of this section, set $\bbone^{[0]} = \bbone$, and for $\sfd > 0$ let $\bbone^{[\sfd]}$ be defined by
\begin{gather}
  \bbone
\lhra
  \bbone^{[\sfd]}
\thra
  \bbone^{[\sfd-1]}
  \otimes
  \rmH^1( \Gamma,\, \bbone )
\tx{,}
\end{gather}
where
\begin{gather*}
  \rmH^1( \Gamma,\, \bbone )
\lhra
  \rmH^1\big( \Gamma,\, \bbone^{[\sfd-1]} \big)
\cong
  \Ext^1_\Gamma\big( \bbone,\, \bbone^{[\sfd-1]} \big)
\tx{.}
\end{gather*}
Clearly, $\bbone^{[\sfd]}$ is real-arithmetic and its socle factors are isotrivial.

The next result yields case~\ref{it:thm:main:recursion-rationality:higher-order} of Theorem~\ref{thm:main:recursion-rationality}.
\begin{proposition}
\label{prop:higher-order-modular-forms}
Let $\sfd \ge 0$ and $k \in \ZZ$. Then the map
\begin{gather}
\label{eq:prop:higher-order-modular-forms}
  \rmM_k\big( \bbone^{[\sfd]} \big)
\lra
  \rmM^{[\sfd]}_k
\tx{,}\quad
  f \boxplus \ast
\lmto
  f
\end{gather}
is surjective. In particular, we have
\begin{gather*}
  \rmM_k\big( \bbone^{[\sfd]} \big)
\;\cong\;
  \bigoplus_{j = 0}^\sfd
  \rmM^{[j]}_k \otimes \rmH^1(\Gamma, \bbone)^{\otimes (\sfd-j)}
\tx{.}
\end{gather*}
\end{proposition}
\begin{proof}
The second statement in the proposition follows from the first one in a straightforward way. We therefore focus on establishing surjectivity of the first map.

If $\sfd = 0$, the statement is vacuous. We use induction to establish the statement if~$\sfd > 0$. Assume that
\begin{gather*}
  \rmM_k\big( \bbone^{[\sfd-1]} \big)
\lra
  \rmM^{[\sfd-1]}_k
\tx{,}\quad
  f \boxplus \ast
\lmto
  f
\end{gather*}
is surjective, and fix $h \in \rmM^{[\sfd]}_k$. By definition of higher order modular forms we have
\begin{gather*}
  \varphi_h(\gamma)
=
  h \big|_k (\gamma - 1)
\tx{,}\quad
  \varphi_h
\in
  \rmH^1\big( \Gamma,\, \rmM^{[\sfd-1]}_k \big)
\;\cong\;
  \rmM^{[\sfd-1]}_k
  \otimes
  \rmH^1( \Gamma, \bbone )
\tx{.}
\end{gather*}
Employing surjectivity of the map from $\rmM_k(\bbone^{[\sfd-1]})$ onto $\rmM_k^{[\sfd-1]}$, we obtain a tensor
\begin{gather*}
  \varphi_h
\in
  \rmM_k\big( \bbone^{[\sfd-1]} \big)
  \otimes
  \rmH^1( \Gamma, \bbone )
\tx{.}
\end{gather*}
In particular, for any $f\in M_k^{[\sfd]}$, we get 
\begin{gather*}
  \big( f \,\boxplus\, \varphi_h \big) \big|_k \gamma
=
  \big( f + \varphi_h(\gamma) \big)
  \,\boxplus\,
  \varphi_h
=
  \bbone^{[\sfd]}(\gamma)
  \big( f \,\boxplus\, \varphi_h \big)
\tx{.}
\end{gather*}
This provides an inclusion of $\rmM^{[\sfd]}_k$ into $\rmM_k(\bbone^{[\sfd]})$ that is a right inverse to the map in~\eqref{eq:prop:higher-order-modular-forms}, i.e.\@ a section to~\eqref{eq:prop:higher-order-modular-forms}.
\end{proof}

\subsection{Iterated Eichler-Shimura integrals and multiple modular values}
\label{ssec:iterated-integrals}

The study of iterated integrals of modular forms was suggested by Manin in the context of non-commutative modular symbols~\cite{manin-2005,manin-2006}. Brown developed a theory of multi modular values in~\cite{brown-2017-preprint} that features such iterated integrals and connects them to some hypothetical category of motives. The goal of this section is to show that iterated integrals of modular forms are components of vector-valued modular forms of virtually real-arithmetic types. We revisit Brown's regularization of iterated integrals, and then reinterpret Lemma~5.1 of~\cite{brown-2017-preprint}. In fact, the whole section is based on the exposition in~\cite{brown-2017-preprint}.

Manin in his articles referred to iterated integrals of modular forms as iterated Shimura integrals, and Brown follows this convention. Since they generalize Eichler integrals, we prefer to refer to them as iterated Eichler-Shimura integrals.

We focus throughout on the case of iterated integrals of level~$1$ modular forms. This is by no means an essential restriction, but~\cite{brown-2017-preprint} covers only this case. The parts of~\cite{brown-2017-preprint} that we use can apparently be generalized to arbitrary arithmetic types. We have, however, not checked any details.

To begin with, recall that iterated Eichler-Shimura integrals can be attached to a tuple $\ul{f} = (f_1, \ldots, f_\sfd)$, $\sfd \in \ZZ_{\ge 0}$ of modular forms of weights $\ul{k} = (k_1, \ldots, k_\sfd)$ if $f_\sfd$ is cuspidal.
We then set
\begin{gather*}
  I_{\ul{f}}(\tau)
\;:=\;
  \int_\tau^{i \infty}
  \int_{z_1}^{i \infty}
  \cdots
  \int_{z_{\sfd-1}}^{i \infty}\;
  f_1(z_1) \,\cdots\, f_\sfd(z_\sfd)\,
  (\ul{X} - \ul{z})^{\ul{k}-2}\;
  d\!z_\sfd \cdots d\!z_1
\,\in\,
  \Poly\big( \ul{X}, \ul{k}-2 \big)
\tx{,}
\end{gather*}
abbreviating
\begin{gather}
\begin{aligned}
  (\ul{X} - \ul{z} )^{\ul{k}-2}
\;&:=\;
  (X_1 - z_1)^{k_1-2}
  \cdots
  (X_\sfd - z_\sfd)^{k_\sfd-2}
\quad\tx{and}
\\
  \Poly\big( \ul{X}, \ul{k}-2 \big)
\;&:=\;
  \Poly\big( X_1, k_1-2 \big)
  \otimes
  \cdots
  \otimes
  \Poly\big( X_\sfd, k_\sfd-2 \big)
\tx{.}
\end{aligned}
\end{gather}
Given $\ul{k} \in \ZZ^\sfd$, we let
\begin{gather}
\label{eq:def:space-of-iterated-integrals}
  \rmI\rmM_{\ul{k}}
\;:=\;
  \lspan \CC \big\{
  I^\reg_{\ul{f}} \,:\,
  f_1 \in \rmM_{k_1},\ldots,
  f_\sfd \in \rmM_{k_\sfd}
  \big\}
\tx{,}
\end{gather}
where $I^\reg_{\ul{f}}$ is a regularized variant of~$I_{\ul{f}}$, defined below in~\eqref{eq:def:iterated-integral}, that allows us to subsume all $f \in \rmM_{k_\sfd}$. The case of $f_\sfd \in \rmM_{k_\sfd} \setminus \rmS_{k_\sfd}$ will be treated in the following discussion. 

In his definition, Brown found a method to regularize iterated Eichler-Shimura integrals of general modular forms in a way that is arithmetically meaningful. It amounts to a separate integration of the constant term. Details are explained in~\cite{brown-2017-preprint}. Fix $f_1 \in \rmM_{k_1}, \ldots, f_\sfd \in \rmM_{k_\sfd}$ as before, and write
\begin{gather*}
  \sum_{n=0}^\infty
  c(f_i;n)\, e(n \tau)
\tx{,}\quad
  1 \le i \le \sfd
\end{gather*}
for their Fourier expansions. We define
\begin{multline}
\label{eq:def:iterated-integral}
  I^\reg_{\ul{f}}(\tau)
\;:=\;
  \sum_{j = 0}^\sfd \;\;
  \int_\tau^{i \infty}
  \cdots
  \int_{z_{j-1}}^{i \infty}\,
  \int_{z_j}^0
  \cdots
  \int_{z_{\sfd-1}}^0\;
  f_1(z_1) \cdots f_{j-1}(z_{j-1})
  \cdot
  \big( f_j(z_j) - c(f_j;0) \big)
\\
  \cdot\,
  c(f_{j+1};0) \cdots c(f_\sfd;0)\;
  (\ul{X} - \ul{z})^{\ul{k}-2}\;
  d\!z_\sfd \cdots d\!z_1
\tx{,}
\end{multline}
where in the cases $j \in \{0,1\}$ we set $z_{-1} = i\infty$ and $z_0 = \tau$ in order to unify notation. Observe that this integral converges absolutely and locally uniformly, since $f_j(z_j) - c(f_j;0)$ decays exponentially as $z_j \ra i\infty$.

Brown employs a generating function of iterated Eichler-Shimura integrals. In order to relate it to~\eqref{eq:def:iterated-integral}, we have to decode his defining formula. Our notation differs slightly from his so that it matches the present paper's style.

We let $\CC\llangle \rmM_\bullet^\vee[\sym^\bullet] \rrangle$ be the non-commutative algebra
\begin{gather}
\label{eq:def:dual-modular-forms-with-representation}
  \prod_{\sfd = 0}^\infty
  \prod_{k_1,\ldots,k_\sfd}
  \Hom{}_\CC\big(
  \rmM_{k_1} \otimes \cdots \otimes \rmM_{k_\sfd},\;
  \Poly(X_1, k_1-2)
  \otimes \cdots \otimes
  \Poly(X_\sfd, k_\sfd-2)
  \big)
\tx{.}
\end{gather}
Its product structure is given by the outer tensor product in both components of the hom-space. It carries a right action of~$\SL{2}(\RR)$ that arises from the actions on $\Poly(X_i, k_i-2)$. We will later amend it to obtain another action of~$\SL{2}(\ZZ)$, whose semisimplification equals the one given here. We have a grading on $\CC\llangle \rmM_\bullet^\vee[\sym^\bullet] \rrangle$ by the monoid of integer tuples (with concatenation as their multiplication and the length~$0$ tuple as the neutral element). The graded pieces are denoted by~$\CC\llangle \rmM_\bullet^\vee[\sym^\bullet] \rrangle_{\ul{k}}$.

Given tuples $\ul{k}$ and $\ul{k}'$ of length~$\sfd$ and~$\sfd'$, we define the partial ordering $\ul{k}' > \ul{k}$ by $\sfd' > \sfd$ and $(k'_1, \ldots, k'_\sfd) = \ul{k}$. The relation $\ul{k}' \not\equiv \ul{k}$ is defined by $(k'_1, \ldots, k'_\sfd) \ne \ul{k}$ if $\sfd' \ge \sfd$ and $\ul{k}' \ne (k_1, \ldots, k_{\sfd'})$ if $\sfd' \le \sfd$. We obtain families of right-ideals
\begin{gather}
\label{eq:def:dual-modular-forms-with-representation:ideal}
\begin{aligned}
  \CC\llangle \rmM_\bullet^\vee[\sym^\bullet] \rrangle_{> \ul{k}}
&:=
  \prod_{\substack{\sfd' \ZZ_{\ge 0},\, \ul{k}' \in \ZZ_{\ge 0}^{\sfd'} \\ \ul{k}' > \ul{k}}}
  \CC \llangle \rmM_\bullet^\vee[\sym^\bullet] \rrangle_{\ul{k}'},
\\
  \CC\llangle \rmM_\bullet^\vee[\sym^\bullet] \rrangle_{\not\equiv \ul{k}}
&:=
  \prod_{\substack{\sfd' \ZZ_{\ge 0},\, \ul{k}' \in \ZZ_{\ge 0}^{\sfd'} \\ \ul{k}' \not\equiv \ul{k}}}
  \CC \llangle \rmM_\bullet^\vee[\sym^\bullet] \rrangle_{\ul{k}'}
\tx{.}
\end{aligned}
\end{gather}

In Definition~4.4 of~\cite{brown-2017-preprint}, Brown defines a formal generating series~$I(\tau;\infty)$ of regularized iterated Eichler-Shimura integrals, which takes values in~$\CC \llangle \rmM_\bullet^\vee[\sym^\bullet] \rrangle$. We denote it by
\begin{gather*}
  I^\reg_{\mathrm{Brown}}
:\,
  \HS
\lra
  \CC \llangle \rmM_\bullet^\vee[\sym^\bullet] \rrangle
\tx{.} 
\end{gather*}

In what follows we write~$f^\vee$ for the dual of a modular form~$f$ viewed merely as an element of the $\CC$-vector space $\rmM_k$, where $k$ is the weight of~$f$.
\begin{lemma}
\label{la:iterated-integral-brown}
Fix $\ul{k}$ and $\ul{f}$ as above. Let $\pi_{\ul{f}}$ be the projection of~\eqref{eq:def:dual-modular-forms-with-representation} to
\begin{gather*}
  f_1^\vee \otimes \cdots \otimes f_\sfd^\vee
  \,\otimes\,
  \Poly(X_1, k_1 - 2)
  \otimes \cdots \otimes
  \Poly(X_\sfd, k_\sfd - 2)
\tx{.}
\end{gather*}
Then we have $\pi_{\ul{f}} \circ I^\reg_{\mathrm{Brown}} = I^\reg_{\ul{f}}$.
\end{lemma}
\begin{proof}
We have to identify the components of~$I^\reg_{\mathrm{Brown}}$ in Definition~4.4 of~\cite{brown-2017b-preprint}. Brown sets
\begin{gather}
  I^\reg_{\mathrm{Brown}}
:=
  \lim_{z \ra i \infty} \big(
  I_{\mathrm{Brown}}(\tau; z)
  \cdot
  I_{\mathrm{Brown}}^\infty(z; 0)
  \big)
\tx{,}
\end{gather}
where the product is taken with respect to the algebra structure on~$\CC \llangle \rmM^\vee_\bullet[\sym^\bullet] \rrangle$. The definitions of $I_{\mathrm{Brown}}$ and $I_{\mathrm{Brown}}^\infty$ can be found in~(4.3) and~(3.5) of~\cite{brown-2017-preprint}:
\begin{align*}
  I_{\mathrm{Brown}}(\tau; z)
&\;:=\;
  1
  \;+\;
  \sum_{\sfd = 1}^\infty
  \int_\tau^{z}
  \cdots
  \int_{z_{\sfd-1}}^z
  \Omega(z_1) \cdots \Omega(z_\sfd)
\qquad\tx{and}
\\
  I^\infty_{\mathrm{Brown}}(z; 0)
&\;:=\;
  1
  \;+\;
  \sum_{\sfd = 1}^\infty
  \int_z^{0}
  \cdots
  \int_{z_{\sfd-1}}^0
  \Omega^\infty(z_1) \cdots \Omega^\infty(z_\sfd)
\tx{,}
\end{align*}
where
\begin{gather*}
  \Omega(\tau)
:=
  \sum_{k = 0}^\infty
  \sum_{f \in \cB_k}
  f^\vee\;
  f(\tau) \, (X - \tau)^{k-2}\,
  d\!\tau
\quad\tx{and}\quad
  \Omega^\infty(\tau)
:=
  \sum_{k = 0}^\infty
  \sum_{f \in \cB_k}
  f^\vee\;
  c(f;0) \, (X - \tau)^{k-2}\,
  d\!\tau
\tx{.}
\end{gather*}
Here, $\cB = \bigcup_k \cB_k$ is a graded basis of $\rmM_\bullet = \bigoplus_k \rmM_k$ whose elements have rational Fourier coefficients.

Both~\eqref{eq:def:iterated-integral} and the projection~$\pi$ in Lemma~\ref{la:iterated-integral-brown} are linear in $\ul{f}$. In particular, we can and will assume that $f_i \in \cB$ for all $1 \le i \le \sfd$. Then the proof of Lemma~\ref{la:iterated-integral-brown} amounts to simplifying $\pi \big( I^\reg_{\mathrm{Brown}}(\tau) \big)$ in such a way that it equals~\eqref{eq:def:iterated-integral}. Expanding $\pi \big( I^\reg_{\mathrm{Brown}}(\tau) \big)$ we get
\begin{multline*}
  \lim_{z \ra i \infty}
  \sum_{j = 0}^\sfd \;\;
  \int_\tau^z \cdots \int_{z_{j-1}}^z
  f_1(z_1) (X_1 - z_1)^{k_1-2}
  \cdots
  f_j(z_j) (X_j - z_j)^{k_j-2}
  d\!z_j \cdots d\!z_1
\\
  \cdot\,
  \int_z^0 \cdots \int_{z_{\sfd-1}}^0
  c(f_{j+1};0) (X_{j+1} - z_{j+1} )^{k_{j+1}-2}
  \cdots
  c(f_\sfd;0) (X_\sfd - z_\sfd )^{k_\sfd-2} \,
  d\!z_\sfd \cdots d\!z_{j+1}
\tx{.}
\end{multline*}
Rewriting $f_j(z_j) = (f_j(z_j) - c(f_j;0)) + c(f_j;0)$ and then combining the appropriate integrals from~$z_{j-1}$ to~$z$ and from~$z$ to~$0$, we see that if $j \ge 1$, then
\begin{multline*}
  \int_\tau^z \cdots \int_{z_{j-1}}^z
  \int_{z_j}^0 \cdots \int_{z_{\sfd-1}}^0\;
  f_1(z_1) \cdots f_j(z_j)\,
  c(f_{j+1};0) \cdots c(f_\sfd;0)\;
  d\!z_\sfd \cdots d\!z_1
\\
  +\;
  \int_\tau^z \cdots \int_{z_{j-2}}^z
  \int_z^0 \int_{z_j}^0 \cdots \int_{z_{\sfd-1}}^0\;
  f_1(z_1) \cdots f_{j-1}(z_{j-1})\,
  c(f_j;0) \cdots c(f_\sfd;0)\;
  d\!z_\sfd \cdots d\!z_1
\end{multline*}
equals
\begin{multline*}
  \int_\tau^z \cdots \int_{z_{j-1}}^z
  \int_{z_j}^0 \cdots \int_{z_{\sfd-1}}^0\;
  f_1(z_1) \cdots f_{j-1}(z_{j-1})
  \big( f_j(z_j) - c(f_j;0) \big)\,
  c(f_{j+1};0) \cdots c(f_\sfd;0)\;
  d\!z_1 \cdots d\!z_\sfd
\\
  +\;
  \int_\tau^z \cdots \int_{z_{j-2}}^z
  \int_{z_{j-1}}^0 \cdots \int_{z_{\sfd-1}}^0\;
  f_1(z_1) \cdots f_{j-1}(z_{j-1})\,
  c(f_j;0) \cdots c(f_\sfd;0)\;
  d\!z_1 \cdots d\!z_\sfd
\end{multline*}
for $0 \le j \le \sfd$. Using this relation inductively, and then taking the limit $z \ra i \infty$ we obtain Lemma~\ref{la:iterated-integral-brown}.
\end{proof}
 
We next study the transformation behavior of iterated Eichler-Shimura integrals. The transformation behavior of Brown's generating series is stated in Lemma~5.1 of~\cite{brown-2017-preprint}, which says that there is a cocycle
\begin{gather*}
  \varphi
\in
  \rmH^1\big( \SL{2}(\ZZ),\, \CC \llangle \rmM_\bullet^\vee[\sym^\bullet] \rrangle^\times \big)
\end{gather*} 
such that
\begin{gather}
\label{eq:browns-canonical-cocycle}
  I^\reg_{\mathrm{Brown}}(\tau)
=
  I^\reg_{\mathrm{Brown}}(\gamma \tau) \gamma \cdot \varphi(\gamma)
\quad
  \tx{for all $\gamma \in \SL{2}(\ZZ)$.}
\end{gather}
Observe that we obtain an (infinite dimensional) representation of~$\SL{2}(\ZZ)$ by
\begin{gather}
\label{eq:browns-canonical-cocycle-representation}
  \SL{2}(\ZZ)
  \circlearrowright
  \CC \llangle \rmM_\bullet^\vee[\sym^\bullet] \rrangle
\tx{,}\quad
  \gamma v
:=
  v\gamma^{-1} \cdot \varphi\big( \gamma^{-1} \big)
\tx{.}
\end{gather}
We emphasize that this is not the same as the right action~$(v,\gamma) \mto v \gamma$ that we described after~\eqref{eq:def:dual-modular-forms-with-representation}, but rather extends it. We further record for later use that the semisimplification of~\eqref{eq:browns-canonical-cocycle-representation} equals~$(\gamma, v) \mto v \gamma^{-1}$. To see this, observe that the term with grading~$()$, the empty tuple, of $\varphi(\gamma^{-1})$ equals~$1$. It can be computed from the action on the ``lowest piece''~$\Hom(\CC, \CC) \subset \CC \llangle \rmM_\bullet^\vee[\sym^\bullet] \rrangle$.

\begin{proposition}
\label{prop:iterated-integrals-modular}
Given $\ul{k} = (k_1, \ldots, k_\sfd) \in \ZZ^\sfd$, there is a real-arithmetic type~$\rho$ of depth~$\sfd$ such that
\begin{gather*}
  \rmI\rmM_{\ul{k}}
\lhra
  \rmM_0(\rho)
\tx{.}
\end{gather*}

There is a projection $V(\rho) \ra \Poly(\ul{X}, \ul{k}-2)$ such that the associated map $\rmM_0(\rho) \ra \rmC^\infty(\HS \ra \Poly( \ul{X}, \ul{k}-2 ))$ makes the following diagram commutative:
\begin{center}
\begin{tikzpicture}
\matrix(m)[matrix of math nodes,
column sep=1.5em, row sep=1.5em,
text height=1.5em, text depth=1.25ex]{%
\rmI\rmM_{\ul{k}} && \rmM_0(\rho)
\\
& \rmC^\infty\big( \HS \ra \Poly(\ul{X},\ul{k}-2) \big) & \\
\\
};

\path
(m-1-1) edge[right hook->] (m-1-3)
(m-1-1) edge[right hook->] (m-2-2)
(m-1-3) edge[->] (m-2-2)
;
\end{tikzpicture}
\end{center}

\end{proposition}
\begin{remark}
As opposed to the case of mixed mock modular forms and higher order modular forms, we do not state the existence of a retract that is a coordinate projection. This is connected to the following question: Is every function that transforms like an iterated Eichler-Shimura integral actually an iterated Eichler-Shimura integral? It is solved only in the case of usual Eichler integrals. 
\end{remark}
\begin{proof}[{Proof of Proposition~\ref{prop:iterated-integrals-modular}}]
Our first step is to extract from the representation in~\eqref{eq:browns-canonical-cocycle-representation} a finite dimensional one, i.e., an arithmetic type. To this end, recall the ideals~$\CC\llangle \rmM_\bullet^\vee[\sym^\bullet] \rrangle_{> \ul{k}}$ and~$\CC\llangle \rmM_\bullet^\vee[\sym^\bullet] \rrangle_{\not\equiv \ul{k}}$ introduced in~\eqref{eq:def:dual-modular-forms-with-representation:ideal}, which are compatible with the~$\SL{2}(\ZZ)$ action on~$\tau$ in~\eqref{eq:browns-canonical-cocycle-representation}. Indeed, this action only affects the symmetric powers in each factor of~\eqref{eq:def:dual-modular-forms-with-representation}. In particular, we can define an $\SL{2}(\ZZ)$ representation~$\td\rho$ as the quotient of $\CC\llangle \rmM_\bullet^\vee[\sym^\bullet] \rrangle$ by
\begin{gather*}
  \CC\llangle \rmM_\bullet^\vee[\sym^\bullet] \rrangle_{> \ul{k}}
\,+\,
  \CC\llangle \rmM_\bullet^\vee[\sym^\bullet] \rrangle_{\not\equiv \ul{k}}
\tx{.}
\end{gather*}
We obtain a representation space
\begin{multline}
\label{eq:browns-canonical-cocycle-representation-quotient}
  V(\td\rho)
\;:=\;
  \CC
  \;\oplus\;
  \Hom{}_\CC\big( \rmM_{k_1},\; \Poly(X_1, k_1-2) \big)
  \;\oplus\;
\\
  \cdots
  \;\oplus\;
  \Hom{}_\CC\big(
  \rmM_{k_1} \otimes \cdots \otimes \rmM_{k_\sfd},\;
  \Poly(X_1, k_1-2)
  \otimes \cdots \otimes
  \Poly(X_\sfd, k_\sfd-2)
  \big)
\tx{.}
\end{multline}
Observe that a priori this is merely a graded direct sum of vector spaces, since multiplication by~$\varphi(\gamma^{-1})$ in~\eqref{eq:browns-canonical-cocycle-representation} does not respect the direct sum structure. We have given the semisimplification of~$\CC\llangle \rmM_\bullet^\vee[\sym^\bullet]\rrangle$ before, from which we see that the semisimplification of~$\td\rho$ is isomorphic to a direct sum of tensor products of symmetric powers $\sym^{k_1-2} \otimes \cdots \otimes \sym^{k_j-2}$, since we regard the $\rmM_{k_i}$ as trivial $\SL{2}(\ZZ)$ representations. In particular, each of them is real-arithmetic as an arithmetic type. We conclude that $\td\rho$ is also real-arithmetic.

By means of the projection from $\CC\llangle \rmM_\bullet^\vee[\sym^\bullet] \rrangle$ to~\eqref{eq:browns-canonical-cocycle-representation-quotient}, we obtain from $I^\reg_\Brown$ a function~$\ov{I}^\reg_\Brown :\, \HS \ra V(\td\rho)$. The invariance in~\eqref{eq:browns-canonical-cocycle} implies that $\ov{I}^\reg_\Brown \in \rmM_0(\td\rho)$. To match Brown's construction in~\cite{brown-2017-preprint}, fix bases~$f_{k_i,j}$ with $1 \le j \le \dim\,\rmM_{k_i}$ of each space of modular form~$\rmM_{k_i}$. We obtain from the projections~$\pi_{\ul{f}}$ in Lemma~\ref{la:iterated-integral-brown} a linear map defined by
\begin{multline*}
  \pi :\;
  \CC \ov{I}^\reg_\Brown \otimes \rmM_{k_1} \otimes \cdots \otimes \rmM_{k_\sfd}
\lthra
  \rmI\rmM_{\ul{k}}
\tx{,}
\\
  \ov{I}^\reg_\Brown \otimes f_{k_1,j_1} \otimes \cdots \otimes f_{k_\sfd,j_\sfd}
\lmto
  \pi_{\ul{f}} \circ \ov{I}^\reg_\Brown
\tx{,}\quad
  \ul{f}
=
  \big( f_{k_1,j_1}, \ldots,  f_{k_\sfd,j_\sfd} \big)
\tx{.}
\end{multline*}
Since the regularized iterated integral is linear in~$\ul{f}$, Lemma~\ref{la:iterated-integral-brown} implies that~$\pi$ is surjective.

We wish to extract the ``top component'' of $\ov{I}^\reg_\Brown$, corresponding to the last direct summand in~\eqref{eq:browns-canonical-cocycle-representation-quotient}, in order to obtain iterated integrals in the sense of~\eqref{eq:def:space-of-iterated-integrals}. Consider the following diagram of vector space homomorphisms.
\begin{center}
\begin{tikzpicture}
\matrix(m)[matrix of math nodes,
column sep=4em, row sep=4em,
text height=1.5em, text depth=1.25ex]{%
\rmM_{k_1} \otimes \cdots \otimes \rmM_{k_\sfd} &
\CC \ov{I}^\reg_\Brown \otimes \rmM_{k_1} \otimes \cdots \otimes \rmM_{k_\sfd} &
\rmM_0(\td\rho) \otimes \rmM_{k_1} \otimes \cdots \otimes \rmM_{k_\sfd} &
\\
\rmI\rmM_{\ul{k}} &
\rmC^\infty\big( \HS \ra \Poly(\ul{X},\ul{k}-2) \big) &
\rmM_0\big(\td\rho \otimes \rmM_{k_1} \otimes \cdots \otimes \rmM_{k_\sfd} \big) \\
};

\path
(m-1-1) edge[->] node[above] {$\ov{I}^\reg_\Brown \otimes$} (m-1-2)
(m-1-2) edge[right hook->] (m-1-3)

(m-1-1) edge[->>] (m-2-1)
(m-1-2) edge[->>] node[above left] {$\pi$} (m-2-1)
(m-2-1) edge[->, dashed, bend right = 10] node[below right] {$\sigma$} (m-1-2)

(m-1-3) edge[right hook->>] (m-2-3)

(m-2-1) edge[right hook->] (m-2-2)
(m-2-3) edge[->] (m-2-2)
;
\end{tikzpicture}
\end{center}
The left vertical arrow represents the map $\ul{f} \mto I^\reg_{\ul{f}}$, which is surjective by the definition of $\rmI\rmM_{\ul{k}}$---it is in general not an isomorphism, since there might be relations among iterated Eichler-Shimura integrals. The horizontal arrow on the top left is an isomorphism. From the definition of~$\pi$ and Lemma~\ref{la:iterated-integral-brown}, we see that the left triangle is commutative. Since~$\pi$ is surjective, there is a section~$\sigma$ to~$\pi$. This amounts to choosing preimages for each iterated integral in~$\rmI\rmM_{\ul{k}}$. 

The right horizontal arrow on the top arises from the containment~$\ov{I}^\reg_\Brown \in \rmM_0(\td\rho)$. The vertical map on the right-hand side is an isomorphism, that arises from viewing each of the $\rmM_{k_i}$ as an isotrivial arithmetic type. Now, the first part of the proposition follows when choosing for~$\rho$ the representation
\begin{gather*}
  \td\rho \otimes \rmM_0(\td\rho) \otimes \rmM_{k_1} \otimes \cdots \otimes \rmM_{k_\sfd}
\;\cong\;
  \td\rho \otimes \rmM_0\big(\td\rho \otimes \rmM_{k_1} \otimes \cdots \otimes \rmM_{k_\sfd} \big)
\tx{.}
\end{gather*}

To prove the second part of the proposition, it remains to understand the bottom line. The bottom left arrow corresponds to the fact that iterated integrals are smooth. The bottom right arrow arises from the projection
\begin{align*}
&
  V(\td\rho) \otimes \rmM_{k_1} \otimes \cdots \otimes \rmM_{k_\sfd} 
\\
={}\;&
  \Big(
  \CC
  \;\oplus\;
  \Hom{}_\CC\big( \rmM_{k_1},\; \Poly(X_1, k_1-2) \big)
  \;\oplus\;
  \cdots
  \;\oplus\;
  \Hom{}_\CC\big(
  \rmM_{k_1} \otimes \cdots \otimes \rmM_{k_\sfd},\;
  \Poly(\ul{X}, \ul{k}-2)
  \big)
  \Big)
\\
  &\quad
  \otimes \rmM_{k_1} \otimes \cdots \otimes \rmM_{k_\sfd} 
\\
\lthra{}&
  \Hom{}_\CC\big(
  \rmM_{k_1} \otimes \cdots \otimes \rmM_{k_\sfd},\;
  \Poly(\ul{X}, \ul{k}-2)
  \big)
  \otimes \rmM_{k_1} \otimes \cdots \otimes \rmM_{k_\sfd} 
\\
\lthra{}&
  \Poly(\ul{X}, \ul{k}-2)
\tx{.}
\end{align*}
When applying this construction to the image of~$\sigma$, we recognize the definition of~$\pi$. In particular, the commutativity of the diagram in Proposition~\ref{prop:iterated-integrals-modular} follows from Lemma~\ref{la:iterated-integral-brown}.
\end{proof}

\section{Existence of modular forms}
\label{sec:existence-of-modular-forms}

Existence of modular forms is typically established via cohomological arguments or via Poincar\'e series constructions. In the case of mock modular forms cohomological arguments can be found in~\cite{bruinier-funke-2004}, and the Poincar\'e series construction (without analytic continuation) was used in~\cite{bringmann-ono-2007} recasting results of~\cite{niebur-1973}.

\subsection{A cohomological argument}

Consider an arithmetic type $\rho$ that is an extension $\rho_\iota \hra \rho \thra \rho_\pi$. Then there are natural maps
\begin{gather*}
  \rmM_k(\rho)  \lra \rmM_k(\rho_\pi)
\quad\tx{and}\quad
  \rmM^!_k(\rho)  \lra \rmM^!_k(\rho_\pi)
\tx{.}
\end{gather*}
It is natural to ask whether or not these maps are surjective. The first one in general is not, but the second one is. Obstructions to surjectivity are given by suitable first cohomology groups, which by Serre duality are related to modular forms of sufficiently high vanishing order. The proof of the next theorem exploits this connection.
\begin{theorem}
\label{thm:existence-of-modular-forms}
Let
\begin{gather*}
  \rho_\iota \lhra \rho \lthra \rho_\pi
\end{gather*}
be an exact sequence of arithmetic types. Then the map
\begin{gather*}
  \rmM^!_k( \rho )
\lra
  \rmM^!_k( \rho_\pi )
\end{gather*}
is surjective. The map
\begin{gather*}
  \rmM_k( \rho )
\lra
  \rmM_k( \rho_\pi )
\end{gather*}
is surjective if and only if\/ $\dim\, \rmM_{2-k}(\rho_\iota^\vee) = 0$.
\end{theorem}

\begin{remark}
Before we prove Theorem~\ref{thm:existence-of-modular-forms}, we elaborate briefly on the relation to Bruinier-Funke's existence theorem for harmonic weak Maa\ss\ forms. Recall from the theory of harmonic weak Maa\ss\ forms the $\xi$-operators $\xi_k\,f := - 2 i y^k \ov{\partial_{\ov \tau} f}$. It maps harmonic weak Maa\ss\ forms $f$ of weight~$k$ to (weakly) holomorphic modular forms of weight $2-k$. Theorem~3.7 of~\cite{bruinier-funke-2004} says that $\xi_k$ is surjective for every~$k$ onto modular forms of weight~$2-k$ when allowing meromorphic singularities at the cusp of the preimage. Bruinier-Funke's method of proof is parallel to the one that we employ to establish Theorem~\ref{thm:existence-of-modular-forms}, but it seems worthwhile to emphasize the difference between the statements. Recall from Corollary~\ref{cor:mock-modular-as-universal-type-modular-form} the map
\begin{gather*}
  \bbM_{-\sfd}(\rho) \lra \rmM^!_{-\sfd}(\rho \extbpara \sym^\sfd)
\tx{,}\quad
  f
\lmto
  f \,\boxplus\, (X - \tau)^\sfd \otimes \varphi_f
\tx{.}
\end{gather*}
is injective. By definition of $\rho \extbpara \sym^\sfd$, there is an exact sequence
\begin{gather*}
  \rho
\lhra
  \rho \extbpara \sym^\sfd
\lthra
  \sym^\sfd \otimes \Extpara(\rho, \bbone)
\tx{.}
\end{gather*}
Composing the above embedding of mock modular forms with the second map in this
sequence, we obtain a map
\begin{gather*}
  \bbM_{-\sfd}(\rho)
\lra
  \rmM^!_{-\sfd}(\rho \extbpara \sym^\sfd)
\lra
  \rmM^!_{-\sfd}(\sym^\sfd)
  \otimes
  \Extpara(\rho, \bbone)
\tx{,}
\end{gather*}
which sends $f$ to $(X - \tau)^\sfd \otimes \varphi_f$. Proposition~\ref{prop:embedding-of-modular-forms} informs us about modular forms of symmetric power type. In particular, we have a projection $\rmM^!_{-\sfd}(\sym^\sfd) \thra \rmM^!_0$. When composing it with the above map, we obtain
\begin{gather*}
  \bbM_{-\sfd}(\rho)
\lra
  \rmM^!_{-\sfd}(\rho \extbpara \sym^\sfd)
\lra
  \rmM^!_{-\sfd}(\sym^\sfd)
  \otimes
  \Extpara(\rho, \bbone)
\thra
  \rmM^!_0
  \otimes
  \Extpara(\rho, \bbone)
\tx{.}
\end{gather*}
The image of $f$ under this map is $1 \otimes \varphi_f = \varphi_f$. The shadow of $f$ is encoded as a cocycle $\varphi_f$ via the Eichler-Shimura theorem, but not as a modular form as in the theory of Bruinier-Funke. Theorem~\ref{thm:existence-of-modular-forms} suggests that we study modular forms of virtually real-arithmetic beyond the image of the maps in~\eqref{eq:cor:mixed-mock-modular-as-universal-type-modular-form} and~\eqref{eq:cor:mixed-mock-modular-as-couniversal-type-modular-form}.
\end{remark}

\begin{proof}[{Proof of Theorem~\ref{thm:existence-of-modular-forms}}]
Without loss of generality, we can pass from $\Gamma$ to a torsion-free subgroup of finite index (cf.~\cite{mennicke-1967,mennicke-1968}). In particular, we can and will assume that $\Gamma \backslash \HS$ is a manifold. Adding cusps to $\Gamma \backslash \HS$, we obtain a compactification $\ov{\Gamma \backslash \HS}$. The structure sheaf~$\cO$ of $\Gamma \backslash \HS$ extends to~$\ov{\Gamma \backslash \HS}$. We let $\cO(m) := \cO(m\, \partial (\Gamma \backslash \HS))$ be the twist of~$\cO$ by a multiple of the boundary divisor $\partial (\Gamma \backslash \HS)$ of $\Gamma \backslash \HS$. Any compatible choice of logarithms of $\rho(T)$ for parabolic $T \in \Gamma$ yields local systems $\cV_k(\rho)$ on $\ov{\Gamma \backslash \HS}$ attached to weight~$k$ and arithmetic type~$\rho$. We fix the choice of $\log\,\rho(T)$ that corresponds to holomorphic components of local sections. Details are conveniently given in the context of Proposition~3.2 of~\cite{candelori-franc-2017}.

Observe that the exact sequence of types $\rho_\iota \hra \rho \thra \rho_\pi$ corresponds to an exact sequence of local systems
\begin{gather*}
  \cV_k(\rho_\iota)
\lhra
  \cV_k(\rho)
\lthra
  \cV_k(\rho_\pi)
\tx{.}
\end{gather*}

Twisting $\cV_k(\rho)$ by $\cO(m)$ we find that
\begin{gather*}
  \rmM^!_k(\rho)
\cong
  \injlim_{m \,\ra\, \infty}
  \rmH^0\big( 
  \ov{\Gamma \backslash \HS},\,
  \cV_k(\rho) \otimes \cO(m)
  \big)
\tx{.}
\end{gather*}
The analogue holds for weakly holomorphic modular forms of arithmetic type~$\rho_\pi$. In order to establish the first claim, it therefore suffices to show that
\begin{gather*}
  \rmH^0\big(
  \ov{\Gamma \backslash \HS},\,
  \cV_k(\rho) \otimes \cO(m)
  \big)
\lra
  \rmH^0\big(
  \ov{\Gamma \backslash \HS},\,
  \cV_k(\rho_\pi) \otimes \cO(m)
  \big)
\end{gather*}
is surjective, if~$m$ is large enough.

We have an exact sequence of holomorphic sheaves on $\ov{\Gamma \backslash \HS}$:
\begin{gather}
\label{eq:thm:existence-of-modular-forms:sheaf-sequence}
  \cV_k(\rho_\iota) \otimes \cO(m)
\lhra
  \cV_k(\rho) \otimes \cO(m)
\lthra
  \cV_k(\rho_\pi) \otimes \cO(m)
\tx{.}
\end{gather}
In the following sequence, we suppress $\ov{\Gamma \backslash \HS}$ from the notation. The first four terms of the long exact sequence in de Rham cohomology associated to~\eqref{eq:thm:existence-of-modular-forms:sheaf-sequence} are
\begin{gather*}
  0
\lra
  \rmH^0\big(
  \cV_k(\rho_\iota) \otimes \cO(m)
  \big)
\lra
  \rmH^0\big(
  \cV_k(\rho) \otimes \cO(m)
  \big)
\lra
  \rmH^0\big(
  \cV_k(\rho_\pi) \otimes \cO(m)
  \big)
\lra
  \rmH^1\big(
  \cV_k(\rho_\iota) \otimes \cO(m)
  \big)
\tx{.}
\end{gather*}
By Serre duality, we have
\begin{gather*}
  \rmH^1\big(
  \cV_k(\rho_\iota) \otimes \cO(m)
  \big)
\;\cong\;
  \rmH^0\big(
  \cV_{2-k}(\rho_\iota^\vee) \otimes \cO(-m)
  \big)^\vee
\tx{,}
\end{gather*}
since the canonical sheaf on~$\ov{\Gamma \backslash \HS}$ is isomorphic to~$\cV_{2}$. The right-hand side vanishes, if~$m$ is sufficiently large. This implies the first claim. The second claim follows when considering the case of~$m = 0$.
\end{proof}

\subsection{Eisenstein series and Poincar\'e series}
\label{ssec:eisenstein-poincare-series}

We prove that holomorphic Eisenstein series and Poincar\'e series of sufficiently large weight converge and span the space of modular forms. We also give an explicit bound on the weight of convergence. Both~\cite{knopp-mason-2011} and~\cite{fedosova-pohl-2017-preprint} contain a proof of convergence for much more general types. A bound on the weight of convergence in terms of geodesics is given in~\cite{fedosova-pohl-2017-preprint}. A meromorphic continuation for Eisenstein series of vra~type has not yet been delivered to the authors' knowledge. Already for arithmetic types of socle length greater than~$1$, the theory becomes more complicated. Evidence that these additional complications are substantial can be extracted from O'Sullivan's analytic continuation of second order Eisenstein series~\cite{osullivan-2000} in conjunction with Section~\ref{ssec:higher-order-modular-forms} of the present paper. O'Sullivan found infinitely many poles of the weight~$0$ second order Eisenstein series, whose location is determined by the discrete spectrum of the Laplacian on $\rmL^2(\Gamma \backslash \HS)$.

For simplicity, we work with vra~types~$\rho$ of $\Gamma = \SL{2}(\ZZ)$. Let $\psi \in \Poly(\tau) \otimes V(\rho)$ and $m \in \QQ$ such that $\psi e(m \,\cdot\,) |_{k,\rho}\, T = \psi e(m \,\cdot\,)$. We define the associated Poincar\'e series as
\begin{gather}
\label{eq:def:poincare-series}
  P_{k,\rho}\big( \psi e(m \,\cdot\,) \big)
:=
  \sum_{\Gamma_\infty \backslash \Gamma} \psi e(m \,\cdot\,) \big|_{k,\rho}\, \gamma
\tx{.}
\end{gather}
For simplicity, we write
\begin{gather}
  P_{k,\rho}\big(\psi e(m \,\cdot\,);\, \tau \big)
:=
  P_{k,\rho}\big(\psi e(m \,\cdot\,)\big) (\tau)
\tx{.}
\end{gather}
If $m = 0$, Poincar\'e series are Eisenstein series, and we set $E_{k,\rho}(\psi;\, \tau) := P_{k,\rho}(\psi e(0 \,\cdot\,);\, \tau)$. For convenience, we let $P_{k,\rho}\big( \psi e(m \,\cdot\,) \big) = 0$ if $\psi e(m \,\cdot\,) |_{k,\rho}\, T \ne \psi e(m \,\cdot\,)$.
\begin{remark}
The Jordan normal form for matrices over $\CC$ implies that twists of $\Poly(\tau)$ by characters of the group~$T^\bullet = \{ T^n \,:\, n \in \ZZ \} \subset \SL{2}(\ZZ)$ exhaust all indecomposable, finite-di\-men\-sion\-al, complex representations of~$T^\bullet$. We will see in Proposition~\ref{prop:poincare-series-generate} that this can be employed to obtain a generating set of $\rmM_k(\rho)$ in complete analogy to holomorphic Poincar\'e series that span $\rmM_k$ for large enough~$k$.
\end{remark}

\begin{proposition}
\label{prop:poincare-series-convergence}
For $k \ge 5 + \shift(\rho) + \pxs(\rho)$, the right-hand side of~\eqref{eq:def:poincare-series} converges absolutely and locally uniformly.
\end{proposition}
Before we establish Proposition~\ref{prop:poincare-series-convergence}, we give the statement of main interest:
\begin{proposition}
\label{prop:poincare-series-generate}
Suppose that $k \ge 5 + \shift(\rho) + \pxs(\rho)$. Then
\begin{gather*}
  \rmM_k(\rho)
\;=\;
  \lspan \CC\big\{
  P_{k,\rho}\big(\psi e(m \,\cdot\,);\, \tau \big)
  \,:\,
  0 \le m \in \QQ,\,
  \psi \in \Poly(\tau) \otimes V(\rho),\,
  \psi e(m \,\cdot\,) |_{k,\rho}\, T = \psi e(m \,\cdot\,)
  \big\}
\end{gather*}
\end{proposition}
\begin{proof}[{Proof of Proposition~\ref{prop:poincare-series-generate}}]
Using induction on the socle length we can focus on semisimple arithmetic types. In other words, we can assume that $\rho = \sym^l$. Then Proposition~\ref{prop:embedding-of-modular-forms} with~$\rho \leadsto \bbone$ shows that the statement is implied by the case of trivial~$\rho$, i.e., $l = 0$. The proof is hence reduced to the classical statement that Poincar\'e series span the space of scalar-valued modular forms.
\end{proof}

The proof of Proposition~\ref{prop:poincare-series-convergence} is standard, except for bounds on the operator norm of real-arithmetic types. They are stated in the next lemma, whose proof requires the following notation: Any rational number~$\alpha$ can be expanded as a minus continued fraction
\begin{gather}
\label{eq:minus-continued-fraction-expansion}
  \alpha
=
  \alpha_0
  \,-\, \frac{1}{\alpha_1}
  \,\genfrac{}{}{0pt}{}{}{-}\, \frac{1}{\alpha_2}
  \,\genfrac{}{}{0pt}{}{}{-}\, 
  \;\cdots\;
  \,\genfrac{}{}{0pt}{}{}{-}\, \frac{1}{\alpha_{l(\alpha)}}
=:
  \llparen \alpha_0, \alpha_1, \alpha_2, \ldots, \alpha_{l(\alpha)} \rrparen
\tx{.}
\end{gather}
As in the case of the usual continued fraction expansion, the length of the minus continued fraction expansion of $\alpha = \tfrac{d}{c}$ grows at most logarithmically: $l(\alpha) \ll \log(|c| + |d|)$.

Given $\gamma = \begin{psmatrix} a & b \\ c & d \end{psmatrix} \in \SL{2}(\ZZ)$ with $c \ne 0$, then
\begin{gather*}
  \gamma
=
  \pm
  T^{m(\gamma)}
  S T^{\alpha_l} \cdots S T^{\alpha_0}
\quad
  \tx{for $\tfrac{d}{c} = \llparen \alpha_0, \ldots, \alpha_l \rrparen$ and some $m(\gamma) \in \ZZ$.}
\end{gather*}
\begin{lemma}
\label{la:estimate-cocycle-and-representation-norm}
Let $\rho_i$, $1 \le i \le i_0$ be vra~types of\/ $\SL{2}(\ZZ)$, and set $\rho = \bigotimes \rho_i$. Let $\varphi \in \rmH^1(\SL{2}(\ZZ), \rho)$, and fix some norm~$\|\,\cdot\,\|$ on~$V(\rho)$. Denote the associated operator norms by $\|\,\cdot\,\|$, too. There is $\kappa > 3$ such that for every~$\gamma \in \SL{2}(\ZZ)$, we have
\begin{align}
\label{eq:estimate-representation-norm}
  \| \rho(T^{-m(\gamma)} \gamma) \|
\ll_\rho
  \big( |c| + |d| \big)^{i_0 \kappa + \sum \shift(\rho_i) + \sum \pxs(\rho_i)}
\quad\tx{and}\\
\label{eq:estimate-cocycle-norm}
  \| \varphi(T^{-m(\gamma)} \gamma) \|
\ll_\rho
  \big( |c| + |d| \big)^{i_0 \kappa + \sum \shift(\rho_i) + \sum \pxs(\rho_i) + \pxs(\varphi)}
\tx{,}
\end{align}
uniformly in $\gamma$.
\end{lemma}
\begin{proof}
We prove the lemma by induction on the socle length. In the case that all $\rho_i$ have socle length~$1$, it suffices to treat a single irreducible representation $\rho = \sym^j$, since $\bigotimes \rho_i$ is a direct sum of symmetric powers whose weight shifts are bounded by $\shift(\rho) = \sum \shift(\rho_i)$.

To establish the base case with respect to the induction on the socle length, we employ induction on the length of the minus continued fraction expansion~\eqref{eq:minus-continued-fraction-expansion} of $d \slash c$. If $c = 0$, then $\gamma = \pm T^{m(\gamma)}$. In case that $\varphi$ is parabolic, we have $\varphi(T^m) = 0$. For general $\varphi$ and $m > 0$, we find that
\begin{gather*}
  \Big\| \varphi\big(T^m\big) \Big\|
=
  \sum_{n = 0}^{m-1}
  \Big\| \sym^j\big(T^{-n}\big) \varphi(T) \Big\|
\ll_\rho
  \sum_{n = 0}^{m-1}
  (1 + n)^j
\ll_\rho
  (1 + m)^{j+1}
\tx{.}
\end{gather*}
An analogous estimate holds if $m < 0$. The following thus holds for all~$m \in \ZZ$:
\begin{gather*}
  \Big\| \varphi\big(T^m\big) \Big\|
\ll_\rho
  (1 + |m|)^{j + \pxs(\varphi)}
\tx{.}
\end{gather*}
In conjunction with $\varphi(\pm 1) \ll_\rho 1$, this settles the base case of the second induction.

Still restricting to the case of $\rho = \sym^j$, we assume that~\eqref{eq:estimate-cocycle-norm} holds for all $\gamma$ such that $c \ne 0$ and $d \slash c$ has a minus continued fraction expansion of length at most~$l-1$. Write $d = n c + f$ with $|f| < |c|$. Then $\gamma = \gamma' S T^n$ for some $\gamma'$ that satisfies the induction hypothesis. We have that
\begin{multline*}
  \Big\| \varphi\big( \gamma' S T^n \big) \Big\|
\le
  \Big\| \sym^j\big(T^{-n} S^{-1}\big)\, \varphi(\gamma') \Big\|
  +
  \Big\| \varphi\big(S T^n\big) \Big\|
\ll_\rho
  (1 + |n|)^j
  \big\| \varphi( \gamma' ) \big\|
  +
  (1 + |n|)^{j + \pxs(\varphi)}
\\
\ll_\rho
  (1 + |n|)^j
  \big( |c| + |f| \big)^{j + \pxs(\varphi)}
  +
  (1 + |n|)^{j + \pxs(\varphi)}
\le
  4 \big( |c| + |d| \big)^{j + \pxs(\varphi)}
\tx{.}
\end{multline*}
We therefore find that
\begin{gather*}
  4^{l(d \slash c)}\, \big( |c| + |d| \big)^{j + \pxs(\varphi)}
\ll_\rho
  \big( |c| + |d| \big)^{\kappa + j + \pxs(\varphi)}
\end{gather*}
for suitable~$\kappa > \log_{(\sqrt{5} + 1) \slash 2}(4) \approx 2.88$. This establishes~\eqref{eq:estimate-cocycle-norm} if $\rho = \sym^l$. The estimate~\eqref{eq:estimate-representation-norm} for $\| \rho(\gamma) \|$ follows along the same lines.

Now consider real-arithmetic types $\rho_i$ of socle lengths $s_1 \ge \cdots \ge s_j$, and assume that we have established the lemma for types of socle lengths $s_1 - 1, s_2, \ldots, s_j$. Writing $s = s_1$, let $\rho_{1,1} \subset \cdots \subset \rho_{1,s} = \rho_1$ be the socle series of $\rho_1$. The highest socle factor $\rho_1 \slash \rho_{1,s-1}$ will be denoted by $\ov{\rho}_1$. Then $\rho_1$ is an extension of $\ov{\rho}_1$ by $\rho_{1,s-1}$, and we write
\begin{gather*}
  \ov{\varphi}_1
\in
  \rmH^1\big( \SL{2}(\ZZ),\, \rho_{1,s-1} \otimes \ov{\rho}_1^\vee \big)
\cong
  \Ext^1_{\SL{2}(\ZZ)} \big( \rho_{1,s-1}, \ov{\rho}_1 \big)
\end{gather*}
for the associated cocycle. For any $v = (v_{1,s-1}, \ov{v}_1) \otimes v_2 \otimes \cdots \otimes v_j \in V(\rho)$, we find that
\begin{multline*}
  \big\| \rho(\gamma) v \big\|
\;\ll\;
  \big\|
  \ov{\rho}_1(\gamma) \ov{v}_1
  + \ov{\varphi}_1(\gamma) v_{1,s-1}
  + \rho_{1,s-1}(\gamma) v_{1,s-1}
  \big\|\;
  \big\| \rho_2(\gamma) v_2 \big\| \cdots
  \big\| \rho_j(\gamma) v_j \big\|
\\
\ll
  \Big(
  \big(|c| + |d|\big)^{\kappa + \shift(\ov{\rho}_1)}
  +
  \big(|c| + |d|\big)^{\kappa + \shift(\ov{\rho}_1) + \shift(\rho_{1,s-1}) + \pxs(\rho_{1,s-1}) + \pxs(\ov{\varphi}_1)}
\\
  +
  \big(|c| + |d|\big)^{\kappa + \shift(\rho_{1,s-1}) + \pxs(\rho_{1,s-1})}
  \Big)
  \cdot\, \big(|c| + |d|\big)^{\kappa + \sum_{i \ge 2} \shift(\rho_i) + \sum_{i \ge 2} \pxs(\rho_i)}
\tx{.}
\end{multline*}
We obtain the estimate for $\rho(\gamma)$ after simplifying the right-hand side. The estimate for $\varphi(\gamma)$ follows similarly.
\end{proof}

\renewbibmacro{in:}{}
\renewcommand{\bibfont}{\normalfont\small\raggedright}
\renewcommand{\baselinestretch}{.8}

\Needspace*{4em}
\begin{multicols}{2}
\printbibliography[heading=none]%
\end{multicols}

\Needspace*{3em}
\noindent
\rule{\textwidth}{0.15em}

{\noindent\small
University of Liverpool,
Department of Mathematical Sciences, 
Mathematical Sciences Building,
Liverpool, L69 7ZL, UK\\
E-mail: \url{m.h.mertens@liverpool.ac.uk}
}\vspace{.5\baselineskip}

{\noindent\small
Chalmers tekniska högskola och G\"oteborgs Universitet,
Institutionen för Matematiska vetenskaper,
SE-412 96 Göteborg, Sweden\\
E-mail: \url{martin@raum-brothers.eu}\\
Homepage: \url{http://raum-brothers.eu/martin}
}%

\end{document}

